\documentclass[english,leqno,10pt,a4paper]{amsart}
\usepackage{amsmath,amstext, amsthm, amssymb}

\usepackage{pifont}
\usepackage{hyperref}
\usepackage{srcltx}
\usepackage[top=3.5 true cm, left=2.5 true cm, totalheight=23 true cm, totalwidth=16 true cm]{geometry}
\usepackage[english]{babel}
\usepackage[all]{xy}

\usepackage{t1enc}

\newtheorem{thm}{Theorem}[section]
\newtheorem{cor}[thm]{Corollary}
\newtheorem{lemma}[thm]{Lemma}
\newtheorem{prop}[thm]{Proposition}
\newtheorem{conjdefn}[thm]{Conjectural definition}

\newtheorem{conj}[thm]{Conjecture}

\theoremstyle{remark}
\theoremstyle{definition}
\newtheorem{defn}[thm]{Definition}

\newtheorem{rmk}[thm]{Remark}
\newtheorem{exa}[thm]{Example}

\numberwithin{equation}{thm}

\def\beq{\begin{equation}}
\def\eeq{\end{equation}}

\def\crash#1{}

\def\N{{\mathbb N}}
\def\Z{{\mathbb Z}}
\def\Q{{\mathbb Q}}
\def\R{{\mathbb R}}
\def\C{{\mathbb C}}

\def\l{\left}
\def\r{\right}
\def\[[{\l[\l[}
\def\]]{\r]\r]}
\def\p{\prime}

\def\sgq{\sigma_q}
\def\sgp{\sigma_p}
\def\Sgq{\Sigma_q}
\def\dq{d_q}

\def\Cup{\mathop\cup}
\def\ord{{\rm ord}}

\def\cf{\emph{cf. }}
\def\ie{\emph{i.e. }}
\def\ds{\displaystyle}

\def\cA{{\mathcal A}}
\def\cB{{\mathcal B}}
\def\cC{{\mathcal C}}

\def\cF{{\mathcal F}}
\def\cM{{\mathcal M}}
\def\cN{{\mathcal N}}

\def\cL{{\mathcal L}}
\def\cO{{\mathcal O}}
\def\cP{{\mathcal P}}

\def\cU{{\mathcal U}}

\def\cY{{\mathcal Y}}
\def\cZ{{\mathcal Z}}

\def\wtilde{\widetilde}

\def\ul{\underline}
\def\ol{\overline}

\def\veps{\varepsilon}
\def\a{\alpha}

\def\sg{\sigma}

\def\la{\lambda}
\def\La{\Lambda}

\def\Sg{\Sigma}
\def\De{\Delta}

\def\lgp{\log^+}
\def\deg{\mathop{\rm deg}}
\def\adj{\mathop{\rm adj}}

\author{Lucia Di Vizio}
\thanks{{Institut de Math\'{e}matiques de Jussieu,
Topologie et g\'{e}om\'{e}trie alg\'{e}briques,}
{Case 7012, 2, place Jussieu, 75251 Paris Cedex 05, France.}
{e-mail: {\tt divizio@math.jussieu.fr}.}
\hfill\break
Work partially supported by ANR, contract ANR-06-JCJC-0028}

\date{\today}

\title[Arithmetic theory of $q$-difference equations]
{Arithmetic theory of $q$-difference equations\\~\\
($G_q$-functions and $q$-difference modules of type $G$, \\
global $q$-Gevrey series)}

\begin{document}
\maketitle
\bibliographystyle{amsalpha}

\begin{abstract}
In the first part of the paper we give a definition of $G_q$-function and
we establish a regularity result, obtained as a combination of a
$q$-analogue of the André-Chudnovsky Theorem \cite[VI]{AGfunctions} and Katz Theorem
\cite[\S13]{KatzTurrittin}. In the second part of the paper, we combine it with some formal
$q$-analogous Fourier transformations, obtaining a statement on the
irrationality of special values of the formal $q$-Borel transformation of
a $G_q$-function.
\end{abstract}

\setcounter{tocdepth}{1}
\tableofcontents

\section{Introduction}

A $G$-function, notion introduced by C.L. Siegel in 1929,
is a formal power series $y=\sum_{n\geq 0}y_n x^n$ with coefficients in the
field of algebraic numbers $\ol\Q$, such that:
\begin{enumerate}
\item
the series $y$ is solution of a linear differential equation with coefficients in $\ol\Q(x)$
(condition that actually ensures that the coefficients of $y$ are contained in a number field $K$);
\item
there exist a sequence of positive numbers $N_n\in\N$ and a positive constant $C$ such that
$N_n y_s$ is an integer of $K$ for any $0\leq s\leq n$ and  $N_n\leq C^n$;
\item
for any immersion $K\hookrightarrow \C$, the image of $y$ in $\C[[x]]$ is a convergent power series for the
usual norm.
\end{enumerate}
Roughly speaking, a $G$-module is a, \emph{a posteriori} fuchsien,
$K(x)/K$-differential module whose (uniform part of)
solutions are $G$-functions (\cf \cite{Bombieri}, \cite{CC1}, \cite{AGfunctions}, \cite{DGS}).
More formally, if $Y^\p(x)=G(x)Y(x)$ is the differential system associate with such a connection in
a given basis, one can iterate it obtaining a family of the higher order differential systems
$\frac1{n!}\frac{d^nY}{dx^n}(x)=G_{[n]}(x)Y(x)$. Our differential module is of type $G$
if there exist a constant $C>0$ and a sequence of polynomials $P_n(x)\in\Z[x]$,
such that
\begin{enumerate}
\item
$P_n(x)G_{[s]}(x)$
is a matrix whose entries are polynomials with coefficients in the ring
of integers of $K$, for any $s=1,\dots,n$;
\item
the absolute value of the coefficients of $P_n(x)$ is smaller that $C^n$.
\end{enumerate}
The unsolved Bombieri-Dwork's conjecture says that $G$-modules \emph{come from geometry},
in the sense that they are extensions of direct summands of Gauss-Manin connections:
the precise conjecture is stated in \cite[II]{AGfunctions}.
Y. André proves that
a differential module coming from geometry is of type $G$ (\cf \cite[V, App.]{AGfunctions}).
More recently, the theory of $G$-functions has been the starting point for the papers \cite{andreannalsI} and
\cite{andreannalsII}, where the author develops an arithmetic theory of Gevrey series,
allowing for a new approach to some diophantine results, such as the Schidlovskii's theorem.
\par
The question of the existence of an arithmetic theory of $q$-difference equations was first asked in
\cite{andreannalsII}.
A naive analogue over a number field of the notion above clearly does not work.
In fact, let $K$ be a number field and let $q\in K$, $q\neq 0$, not be a root of unity.
We consider formal power series $y\in K[[x]]$ that satisfies conditions 2 and 3 of the definition of
$G$-function given above and that is solution of a nontrivial $q$-difference equation
with coefficients in $K(x)$, \ie:
$$
a_\nu(x)y(q^\nu x)+a_{\nu-1}y(q^{nu-1}x)+\dots+a_0(x)y(x)=0\,,
$$
with $a_0(x),\dots,a_\nu(x)\in K(x)$, not all zero.
Then the following result by Y. André is the key point of \cite{DVInv}:

\begin{prop}[{\cite[8.4.1]{DVInv}}]
A series $y$ as above is the Taylor expansion at $0$ of a rational function in $K(x)$.
\end{prop}

Other unsuccessful suggestions for a $q$-analogue of a $G$-function are made in
\cite[App.]{DVTesi}. These considerations may induce to conclude that
$q$-difference equations do not come from geometry over $\ol\Q$.
\par
Here we propose another approach: we consider a finite extension $K$
of the field of
rational function $k(q)$ in $q$ with coefficients in a field $k$.
This is a very natural approach since in the literature, $q$ very often considered as a parameter.
Since $K$ is a global field, we can define a $G_q$-function to be a series in $K[[x]]$, solution
of a $q$-difference equation with coefficients in $K(x)$,
satisfying a straightforward analogue of conditions
2 and 3 of the definition above. As far as the definition of $q$-difference modules of type $G$
is concerned
only the places of $K$ modulo whom $q$ is a root of unity - that we will briefly call cyclotomic places -
comes into the picture (\cf Proposition \ref{prop:cyclotomicplaces} below).
In fact, consider a $q$-difference system
\beq\label{eq:eqintro}
Y(qx)=A(x)Y(x)\,,
\eeq
with $A(x)\in Gl_\nu(K(x))$: its solutions can be interpreted as the horizontal
vectors of a $K(x)$-free module $M$ of rank $\nu$ with respect to a semilinear
bijective operator $\Sgq$ verifying $\Sgq(f(x)m)=f(qx)\Sgq(m)$ for any $f(x)\in K(x)$ and any $m\in M$.
We consider the $q$-derivation:
$$
\dq(f(x))=\frac{f(qx)-f(x)}{(q-1)x}
$$
and its iterations:
$$
\frac{\dq^n}{[n]_q^!}\,,
\hbox{~with $[0]_q^!=[1]_q^!=1$ and $[n]_q^!=\frac{q^n-1}{q-1}[n-1]_q^!$.}
$$
We can obtain from \eqref{eq:eqintro} a whole family of systems:
$$
\frac{\dq^n}{[n]_q^!}Y(x)=G_{[n]}(x)Y(x)\,,
$$
where $G_1(x)=\frac{A(x)-1}{(q-1)x}$ and
$\frac{q^n-1}{q-1}G_{[n]}(x)=G_{[1]}(x)G_{[n-1]}(qx)+\dq G_{[n-1]}(x)$.
The fact that the denominators $[n]_q^!$ of the iterated derivations
$\frac{\dq^n}{[n]_q^!}$ have positive valuation only at the cyclotomic
places has the consequences that ``there is no arithmetic growth''
at the noncyclotomic places (\cf \S\ref{sec:noncyclotomic} below for a precise formulation).
Moreover, an important role in the proofs is played by the reduction
of $q$-difference systems modulo a cyclotomic place: this means that we
specializes $q$ to a root of unity and we study the nilpotence properties
of the obtained system. In characteristic zero, one automatically obtain
an iterative $q$-difference module, in the sense of C. Hardouin \cite{HardouinIterative}.
\par
The role played by the cyclotomic valuations, and therefore by roots of unity,
points out some analogies with other topics:
\begin{itemize}
\item
The Volume Conjecture predicts a link between the hyperbolic volume of the
complement of an hyperbolic knot and the asymptotic of the sequence
$J_n(\exp(2i\pi/n))$, where $J_n(q)$ is an invariant of the knot called
$n$-th Jones polynomial. The Jones polynomials are Laurent polynomials in $q$ such that
the generating series $\sum_{n\geq 0}J_n(q)x^n$ is solution of
a $q$-difference equations with coefficients in $\Q(q)(x)$ (\cf \cite{GarouHolonomy}):
the situation is quite similar to the one considered in the present paper.
The $q$-difference equations appearing in this topological setting
have, in general, irregular singularities, differently from the $q$-difference operators of type $G$,
that are regular singular. To involve some irregular singular operators in the present framework, one
should consider some formal $q$-Fourier transformations
and develop a global theory of $q$-Gevrey series, in the wake of \cite{andreannalsI}:
this is the topic of the second part of the paper.
\item
As already point out, an important role is played by the reduction of $q$-difference systems
modulo the cyclotomic valuations. Conjecturally,
the growth at cyclotomic places should be enough to describe the whole theory
(\cf \S\ref{sec:noncyclotomic}).
It is natural to ask whether $q$-difference equations,
that seem not to ``come from geometry over $\ol\Q$'',
may have some geometric origin, in the sense of the geometry over $\mathbb F_1$
(\cf \cite{Soulecar1}, \cite{connesconsani}).
\end{itemize}
Notice that in \cite{ManinF1}, Y. Manin establish a
link between the Habiro ring, which is a topological algebra
constructed to deal with quantum invariants of knots, and geometry over
$\mathbb F_1$, so that the two remarks above are not orthogonal.

$$
\ast\ast\ast
$$

In the present paper we give a definition of $G_q$-functions and $q$-difference modules of type $G$.
We test those definitions proving that a $q$-difference module having an injective solution whose entries
are $G$-functions is of type $G$: that is to say that ``the minimal $q$-difference module generated by a
$G$-function'' is of type $G$ (\cf Theorem \ref{thm:chudnovsky} below).
We also prove that $q$-difference module of type $G$ are
regular singular (\cf Theorem \ref{thm:regularity}). These two results are the base
for the development of a global theory of $q$-Gevrey series.
\par
In part two, we define global $q$-Gevrey series. Via the study of two $q$-analogues
the formal Fourier transformation, we establish some structure theorems for the minimal $q$-difference
equations killing global $q$-Gevrey series (\cf Theorems \ref{thm:sing}, \ref{thm:singbis} and
\ref{thm:teogev}).  We conclude with an irrationality theorem for special values of of
global $q$-Gevrey series of negative orders (\cf Theorem \ref{thm:irr}).

\medskip
This paper won't be submitted for publication since the results below can be obtained
in a more direct way.
Namely, one can prove that $G_q$-functions are all rational (\cf \cite{DivizioHardouin}).
Nevertheless,
the construction of the coefficients of the $q$-difference module from an injective solution
in the proof of Theorem \ref{thm:chudnovsky}
has an interest in itself, since it may be applied to other difference operators.

\subsection*{Acknowledgements.}
I would like to thank Y. André, J-P. Bézivin and Y. Manin for the interest
they have shown for my work.
Of course, I'm the only responsible for the deficiencies of this paper.

\part{\sc\bfseries $G_q$-functions and $q$-difference modules of type $G$}
\section{Definition and first properties}

Let us consider the field of rational function $k(q)$ with coefficients
in a fixed field $k$.
We fix $d\in(0,1)$ and for any irreducible polynomial $v=v(q)\in k[q]$
we set:
$$
|f(q)|_v=d^{\deg_q v(q)\cdot\ord_{v(q)}f(q)}\,,\,\forall f(q)\in k[q]\,.
$$
The definition of $|~|_v$ extends to $k(q)$ by multiplicativity.
To this set of norms one has to add the $q^{-1}$-adic one, defined on $k[q]$ by:
$$
|f(q)|_{q^{-1}}=d^{-deg_qf(q)}\,;
$$
once again this definition extends by multiplicativity to $k(q)$. Then the
Product Formula holds:
$$
\prod_v\l|\frac{f(q)}{g(q)}\r|_v
=d^{\sum_v\deg_q v(q)~\l(\ord_{v(q)}f(q)-\ord_{v(q)}g(q)\r)}
=d^{deg_q f(q)-\deg_q g(q)}
=\l|\frac{f(q)}{g(q)}\r|_{q^{-1}}^{-1}\,.
$$
For any finite extension $K$ of $k(q)$, we consider the family
$\cP$ of ultrametric norms, that extends the norms defined above, up to equivalence.
We suppose that the norms in $\cP$ are normalized so that the Product Formula still holds.
We consider the following partition of $\cP$:
\begin{itemize}
\item
the set $\cP_\infty$ of places of $K$ such that the associated
norms extend, up to equivalence, either $|~|_q$ or $|~|_{q^{-1}}$;

\item
the set $\cP_f$ of places of $K$ such that the associated
norms extend, up to equivalence, one of the norms $|~|_v$
for an irreducible $v=v(q)\in k[q]$, $v(q)\neq q$.
\end{itemize}
Moreover we consider the set $\cC$ of places $v\in\cP_f$
such that $v$ divides a valuation of $k(q)$ having as uniformizer
a factor of a cyclotomic polynomial. We will briefly call $v\in\cC$ a cyclotomic place.

\begin{defn}
A series $y=\sum_{n\geq 0}y_nx^n\in K[[x]]$ is a \emph{$G_q$-function}
if:
\begin{enumerate}
\item
It is solution of a $q$-difference equations with coefficients
in $K(x)$, \ie there exists $a_0(x),\dots,a_\nu(x)\in K(x)$ not all zero such that
\beq\label{eq:eqdefn}
a_0(x)y(x)+a_1(x)y(qx)+\dots+a_\nu(x)y(q^\nu x)=0\,.
\eeq
\item
The series $y$ has finite size, \ie
$$
\sg(y):=\limsup_{n\to\infty}\frac 1n\sum_{v\in\cP}\log^+
\l(\sup_{s\leq n}|y_s|_v\r)<\infty\,,
$$
where $\log^+ x=\sup(0,\log x)$.
\end{enumerate}
\end{defn}

We will refer to the invariant $\sg$ as the \emph{size}, using the same terminology
as in the classical case
of series over a number field.

\begin{rmk}
\begin{enumerate}
\item
One can show that this definition of $G_q$-function is equivalent to the one given in the
introduction (\cf \cite[I, 1.3]{AGfunctions}).
\item
Let $\ol{k(q)}$ be the algebraic closure of $k(q)$. A formal power series
with coefficients in $\ol{k(q)}$ solution of a $q$-difference equations with coefficients in
$\ol{k(q)}(x)$ is necessarily defined over a finite extension $K/k(q)$.
\end{enumerate}
\end{rmk}

\begin{prop}
The set of $G_q$-functions is stable with respect to the sum and the
Cauchy product\footnote{It may be interesting to remark, although
we won't need it in the sequel, that the estimate of the size of a product of $G$-functions proved in
\cite[I, 1.4, Lemma 2]{AGfunctions} holds also in the case of $G_q$-functions.}.
Moreover, it is independent of the choice of $K$, in the sense that we can replace
$K$ by any finite extension of $K$.
\end{prop}

\begin{proof}
The proof is the same as in the case of classical $G$-functions
(\cf \cite[I, 1.4, Lemma 2]{AGfunctions}).
\end{proof}

The field $K(x)$ is naturally a $q$-difference algebra, \ie is equipped with the
operator
$$
\begin{array}{rccc}
\sgq:&K(x)&\longrightarrow&K(x)\\
&f(x)&\longmapsto&f(qx)
\end{array}\,.
$$
The field $K(x)$ is also equipped with the $q$-derivation
$$
\dq(f)(x)=\frac{f(qx)-f(x)}{(q-1)x}\,,
$$
satisfying a $q$-Leibniz formula:
$$
\dq(fg)(x)=f(qx)\dq(g)(x)+\dq(f)(x)g(x)\,,
$$
for any $f,g\in K(x)$.
A $q$-difference module over $K(x)$ (of rank $\nu$)
is a finite dimensional $K(x)$-vector space $M$ (of dimension $\nu$)
equipped with an invertible $\sgq$-semilinear operator, \ie
$$
\Sgq(f(x)m)=f(qx)\Sgq(m)\,,
$$
for any $f\in K(x)$ and $m\in M$.
A morphism of $q$-difference modules over $K(x)$ is a morphisms
of $K(x)$-vector spaces, commuting to the $q$-difference structure
(for more generalities on the topic, \cf \cite{vdPutSingerDifference},
\cite[Part I]{DVInv}
or \cite{gazette}).
\par
Let $\cM=(M,\Sgq)$ be a $q$-difference module over $K(x)$ of rank $\nu$.
We fix a basis $\ul e$ of $M$ over $K(x)$ and we set:
$$
\Sgq\ul e=\ul e A(x)\,,
$$
with $A(x)\in Gl_\nu(K(x))$.
An horizontal vector $\vec y\in K(x)^\nu$ with respect to $\Sgq$ is a vector
that verifies $\vec y(x)=A(x)\vec y(qx)$. Therefore we call
$$
Y(qx)=A_1(x)Y(x)\,,
\hbox{~with~}A_1(x)=A(x)^{-1}\,,
$$
the system associated to $\cM$ with respect to the basis $\ul e$.
Recursively we obtain the families of $q$-difference systems:
$$
Y(q^nx)=A_n(x)Y(x)
\hbox{~and~}\dq^nY(x)=G_n(x)Y(x)\,,
$$
with $A_n(x)\in Gl_\nu(K(x))$
and $G_n(x)\in M_\nu(K(x))$. Notice that:
$$
A_{n+1}(x)=A_n(qx)A_1(x)\,,\,
G_1(x)=\frac{A_1(x)-1}{(q-1)x}
\hbox{~and~}
G_{n+1}(x)=G_n(qx)G_1(x)+\dq G_n(x)\,.
$$
It is convenient to set $A_0=G_0=1$.
Moreover we set $[n]_q=\frac{q^n-1}{q-1}$ for any $n\geq 1$,
$[n]_q^!=[n]_q[n-1]_q\cdots[1]_q$, $[0]_q^!=1$ and $G_{[n]}(x)=\frac{G_n(x)}{[n]_q^!}$.

\begin{defn}
A $q$-difference module over $K(x)$ is said to be of \emph{type $G$
(or a $G$-$q$-difference module)}
if the following \emph{global $q$-Galo\v ckin condition} is verified:
$$
\sg_\cC^q(\cM)=\limsup_{n\to\infty}\frac 1n\sum_{v\in\cC}\log^+
\l(\sup_{s\leq n}\l|G_{[s]}\r|_{v,Gauss}\r)<\infty\,,
$$
where
$$
\l|\frac{\sum{a_i}x^i}{\sum{b_j}x^j}\r|_{v,Gauss}=\frac{\sup|a_i|_v}{\sup|b_j|_v}\,,
$$
for all $\frac{\sum{a_i}x^i}{\sum{b_j}x^j}\in K(x)$.
\end{defn}

\begin{rmk}
Notice that the definition of $G$-$q$-difference module involves only the cyclotomic places.
\end{rmk}

\begin{prop}
The definition of $G_q$-module is independent on the choice of the basis and is stable by
extension of scalars to $K^\p(x)$, for a finite extension $K^\p$ of $K$.
\end{prop}

\begin{proof}
Once again the proof if similar to the classical theory of $G$-functions and differential modules
of type $G$.
\end{proof}

\section{Role of the ``noncyclotomic'' places}
\label{sec:noncyclotomic}

\begin{prop}\label{prop:cyclotomicplaces}
In the notation introduced above, for any $q$-difference module $\cM=(M,\Sgq)$ over $K(x)$ we have:
$$
\sg^{(q)}_{\cP_f\smallsetminus\cC}(\cM):=
\limsup_{n\to\infty}\frac 1n\sum_{v\in\cP_f\smallsetminus\cC}\log^+
\l(\sup_{s\leq n}\l|G_{[s]}\r|_v\r)<\infty\,.
$$
\end{prop}

\begin{proof}
We recall that the sequence of matrices $G_{[n]}$ satisfies the recurrence relation:
$$
G_{[n+1]}(x)=\frac{G_{[n]}(qx)G_1(x)+\dq G_{[n]}(x)}{[n+1]_q}\,.
$$
Since $|[n+1]_q|_v=1$ for any $v\in\cP_f\smallsetminus\cC$,
we conclude recursively that
$$
\l|G_{[n]}\r|_{v,Gauss}\leq 1\,,
$$
for almost all places $v\in\cP_f\smallsetminus\cC$.
For the remaining finitely many places $v\in\cP_f$, one can deduce from the recursive relation
there exists a constant $C>0$ such that
$\l|G_{[n]}\r|_{v,Gauss}\leq C^n$.
\end{proof}

We immediately obtain the equivalence of our definition of $q$-difference module of type $G$ with
the naive analogue of the classical definition of $G$-module
(\cf \cite[IV, 4.1]{AGfunctions}):

\begin{cor}
A $q$-difference module is of type $G$ if and only if
$$
\sg^{(q)}_{\cP_f}(\cM):=\limsup_{n\to\infty}\frac 1n\sum_{v\in\cP_f}\log^+
\l(\sup_{s\leq n}\l|G_{[s]}\r|_v\r)<\infty\,.
$$
\end{cor}

We expect the same kind of result to be true for $G_q$-functions, namely:

\begin{conj}
Suppose that $y=\sum_{n\geq 0}y_nx^n\in K[[x]]$ is solution of a $q$-difference equations
with coefficients in $K$ (\cf \eqref{eq:eqdefn}).
Then:
$$
\sg_{\cP_f\smallsetminus\cC}(y)
=\limsup_{n\to\infty}\frac 1n\sum_{v\in\cP_f\smallsetminus\cC}\log^+
\l(\sup_{s\leq n}|y_s|_v\r)<\infty\,.
$$
\end{conj}

The last statement would immediately imply that one can define $G_q$-functions
in the following way:

\begin{conjdefn}
We say that the series $y=\sum_{n\geq 0}y_nx^n\in K[[x]]$ is a $G_q$-function if
$y$ is solution of a $q$-difference equations with coefficients in $K$ and moreover
$$
\sg_{\cC\cup\cP_\infty}(y)=\limsup_{n\to\infty}\frac 1n\sum_{v\in\cC\cup\cP_\infty}\log^+
\l(\sup_{s\leq n}|y_s|_v\r)<\infty\,.
$$
\end{conjdefn}

\begin{rmk}
The fact that for almost all $v\in\cP_f\smallsetminus\cC$
we have $|G_{[n]}(x)|_{v,Gauss}\leq 1$ for any $n\geq 1$ implies that for almost all
$v\in\cP_f\smallsetminus\cC$
a ``solution'' $y(x)=\sup_n y_n x^n\in K[[x]]$ of a $q$-difference system with coefficient in $K(x)$ is
bounded, in the sense that $\sup_n|y_n|_v<\infty$.
Unfortunately, one would need some uniformity with respect to $v$ and $n$ to conclude something about
$\sg_{\cP_f\smallsetminus\cC}(y)$.
\par
Notice that if $0$ is an ordinary point, the conjecture is trivial since
$$
\sum_{n\geq 0}G_{[n]}(0)x^n
$$
is a fundamental solution of the linear system $Y(qx)=A_1(x)Y(x)$.
A $q$-analogue of the techniques developed
in \cite[V]{AGfunctions} (\cf also \cite[Chap. VII]{DGS})
would probably allow to establish the conjecture under the assumption that $0$ is a regular point.
This is not satisfactory because one of the purposes of the whole theory is
the possibility of reading the regularity of a $q$-difference equation
on one single solution (\cf Theorem \ref{thm:regularity} below), so one does not want to
assume regularity \emph{a priori}.
\end{rmk}

\section{Main results}

A $q$-difference module $(M,\Sgq)$ is said to be regular singular at $0$ if there
exists a basis $\ul e$ such that the Taylor expansion of the matrix $A_1(x)$
is in $Gl_\nu(K[[x]])$.
It is said to be regular singular \emph{tout court} if it is regular singular
both at $0$ and at $\infty$.
We have the following analogue of a well-known differential result
(\cf \cite[\S13]{KatzTurrittin}; \cf also \cite[\S6.2.2]{DVInv} for $q$-difference modules
over a number field):

\begin{thm}\label{thm:regularity}
A $G$-$q$-difference module $\cM$ over $K(x)$ is regular singular.
\end{thm}

Let $\vec y(x)={}^t(y_0(x),\dots,y_{\nu-1}(x))\in K[[x]]^\nu$
be a solution of the $q$-difference system associated to $\cM=(M,\Sgq)$
with respect to the basis $\ul e$:
$$
\vec y(qx)=A_1(x)\vec y(x)\,.
$$
We say that $\vec y(x)$ is an injective solution if
$y_1(x),\dots,y_\nu(x)$ are lineairly independent over $K(x)$.
\par
We have the following $q$-analogue of the André-Chudnovsky Theorem
\cite[VI]{AGfunctions}:

\begin{thm}\label{thm:chudnovsky}
Let $\vec y(x)={}^t(y_0(x),\dots,y_{\nu-1}(x))\in K[[x]]^\nu$
be an injective solution of the $q$-difference system associated to $\cM=(M,\Sgq)$
with respect to the basis $\ul e$.
\par
If $y_0(x),\dots,y_{\nu-1}(x)$ are $G_q$-functions, then $\cM$ is a $G$-$q$-difference module.
\end{thm}

We can immediately state a corollary:

\begin{cor}
Let $\vec y(x)={}^t(y_0(x),\dots,y_{\nu-1}(x))\in K[[x]]^\nu$
be an injective solution of the $q$-difference system associated to $\cM=(M,\Sgq)$
with respect to the basis $\ul e$.
\par
If $y_1(x),\dots,y_\nu(x)$ are $G_q$-functions, then $\cM$ is regular singular.
\end{cor}

Thanks to the cyclic vector lemma we can state the following
(\cf \cite[Annexe B]{Sfourier}):

\begin{cor}\label{cor:regularity}
Let $y(x)$ a $G_q$-function and let
\beq\label{eq:eq}
a_0(x)y(x)+a_1(x)y(qx)+\dots+a_\nu(x)y(q^\nu x)=0\,.
\eeq
a $q$-difference equation of minimal order $\nu$, having $y(x)$ as a solution.
\par
Then \eqref{eq:eq} is fuchsian,
\ie we have
$\ord_x a_i\geq\ord_x a_0=\ord_x a_\nu$ and $\deg_x a_i\leq\deg_x a_0=\deg_x a_\nu$,  for any $i=0,\dots,\nu$.
\end{cor}

The proofs of Theorem \ref{thm:regularity} and Theorem \ref{thm:chudnovsky}
are the object of \S\ref{sec:proofregularity} and \S\ref{sec:proofchudnovsky}, respectively.

\section{Nilpotent reduction at cyclotomic places}

We denote by $\cO_K$ the ring of integers of $K$, $k_v$ the residue field of $K$
with respect to the pace $v$, $\varpi_v$ the uniformizer of $v$ and $q_v$ the
image of $q$ in $k_v$, which is defined for all places $v\in\cP$.
Notice that $q_v$ is a root of unity for all $v\in\cC$. Let $\kappa_v\in\N$ be the order
of $q_v$, for $v\in\cC$.
\par
Let $\cM=(M,\Sgq)$ be a $q$-difference module over $K(x)$.
We can always choose a lattice $\wtilde M$ of $M$ over an algebra of the form
\beq\label{eq:algebraA}
\cA=\cO_K\l[x,\frac1{P(x)},\frac1{P(qx)},\frac1{P(q^2x)},...\r]\,,
\eeq
for some $P(x)\in\cO_K[x]$,
such that for almost all $v\in\cC$ we can consider
the $q_v$-difference module $M_v=\wtilde M\otimes_\cA k_v(x)$, with the structure induced
by $\Sgq$.
In this way, for almost all $v\in\cC$, we obtain a $q_v$-difference module
$\cM_v=(M_v,\Sg_{q_v})$ over $k_v(x)$, having the particularity that $q_v$ is
a root of unity.  This means that $\sg_{q_v}^{\kappa_v}=1$ and that
$\Sg_{q_v}^{\kappa_v}$ is a $k_v(x)$-linear operator.
\par
The results in \cite[\S2]{DVInv} apply to this situation: we recall some of them.
Since we have:
$$
\sg_{q_v}^{\kappa_v}=1+(q-1)^{\kappa_v} x^{\kappa_v}d_{q_v}^{\kappa_v}
$$
and
$$
\Sg_{q_v}^{\kappa_v}=1+(q-1)^{\kappa_v} x^{\kappa_v}\De_{q_v}^{\kappa_v}\,,
$$
where $\De_{q_v}=\frac{\Sg_{q_v}-1}{(q_v-1)x}$, the following facts are equivalent:
\begin{enumerate}
\item
$\Sg_{q_v}^{\kappa_v}$ is unipotent;
\item
$\De_{q_v}^{\kappa_v}$ is a linear nilpotent operator;
\item
the reduction of $A_{\kappa_v}(x)-1$ modulo $\varpi_v$ is a nilpotent matrix;
\item
the reduction of $G_{\kappa_v}(x)$ modulo $\varpi_v$ is nilpotent;
\item
there exists $n\in\N$ such that $\l|G_{n{\kappa_v}}(x)\r|_{v,Gauss}\leq |\varpi_v|_v$.
\end{enumerate}

\begin{defn}
If the conditions above are satisfied we say that $\cM$ has \emph{nilpotent reduction
(of order $n$)} modulo $v\in\cC$.
\end{defn}

\begin{rmk}
If the characteristic of $k$ is $0$ and if
$\l|G_{\kappa_v}(x)\r|_{v,Gauss}\leq |[\kappa_v]_q|_v$,
the module $\cM_v$ has a structure of iterated $q$-difference module,
in the sense of \cite[\S3]{HardouinIterative}.
In particular, if $v$ is a non ramified place of $K/k(q)$, then $|[\kappa_v]_q|_v=|\varpi_v|_v$.
\end{rmk}

The following result is a $q$-analogue of a well-known differential $p$-adic estimate
(\cf for instance \cite[page 96]{DGS}).
It has already been proved in the case of $q$-difference equations over a
$p$-adic field in
\cite[\S5.1]{DVInv}.
We are
only sketching the argument:
only the estimate of the $q$-factorials are slightly different from the case of mixed characteristic.

\begin{prop}
If $\cM=(M,\Sgq)$ has nilpotent reduction(of order $n$) modulo $v\in\cC$ then
$$
\limsup_{m\to\infty}
\sup\l(1,\l|G_{[m]}\r|_{v,Gauss}\r)^{1/m}\leq
|\varpi_v|_v^{1/n\kappa_n}|[\kappa_v]_q|_v^{-1/\kappa_v}\,.
$$
\end{prop}

\begin{proof}
The Leibniz formula (\cf \cite[Lemma 5.1.2]{DVInv} for a detailed proof in a
quite similar situation) implies that for any
$s\in\N$ we have:
$$
\l|G_{sn{\kappa_v}}(x)\r|_{v,Gauss}\leq |\varpi_v|_v^s\,.
$$
Since $\l|G_1(x)\r|_{v,Gauss}\leq 1$, for any $m\in\N$ we have:
$$
\l|G_{[m]}(x)\r|_{v,Gauss}
\leq
\frac{\l|G_{\l[\frac{m}{n\kappa_v}\r]n\kappa_v}(x)\r|_{v,Gauss}}{|[m]_q^!|_v}
\leq
\frac{\l|\varpi_v\r|_v^{\l[\frac{m}{n\kappa_v}\r]}}{|[m]_q^!|_v}\,,
$$
where $\l[\frac{m}{n\kappa_v}\r]=\max\{a\in\Z: a\leq \frac{m}{n\kappa_v}\}$.
The following lemma on the estimate of $[m]_q^!$
allows to conclude.
\end{proof}

\begin{lemma}\label{lemma:qfactorials}
For $v\in\cC$ we have $|[m]_q|_v=|[\kappa]_q|_v$ if $\kappa_v\vert m$ and $|[m]_q|_v=1$ otherwise.
Therefore:
$$
\lim_{m\to\infty}|[m]_q^!|_v^{1/m}=|[\kappa_v]_q|_v^{1/\kappa_v}\,.
$$
\end{lemma}

\begin{proof}
Let $m\geq 2$ and $m=s\kappa_v+r$, with $r,s\in\Z$ and $0\leq r<\kappa_v$.
If $\kappa_v$ does not divide $m$, \ie if $r>0$, we have
$$
[m]_q=1+q+\dots+q^{m-1}
=[\kappa_v]_q+q^{\kappa_v}[\kappa_v]_q+\dots+q^{s\kappa_v}(1+q+\dots+q^{r-1})\,.
$$
Therefore $|[m]_q|_v=1$.
On the other hand, if $r=0$:
$$
[m]_q=\l(1+q^{\kappa_v}+\dots+q^{\kappa_v(s-1)}\r)[\kappa_v]_q\,.
$$
Since $q^{\kappa_v}\equiv 1$ modulo $\varpi_v$,
we deduce that $1+q^{\kappa_v}+\dots+q^{\kappa_v(s-1)}\equiv s$ modulo $\varpi_v$.
Therefore
$$
\l|[m]_q\r|_v=\l|s\r|_v\l|[\kappa_v]_q\r|_v=\l|[\kappa_v]_q\r|_v\,.
$$
This implies that
$$
\l|[m]_q^!\r|_v=\l|\kappa_v]_q\r|_v^{\l[\frac{m}{\kappa_v}\r]}\,,
$$
which allows to calculate the limit.
\end{proof}

We obtain the following characterization:

\begin{cor}\label{cor:nonnilpextimate}
The $q$-difference module $\cM=(M,\Sgq)$ has nilpotent reduction modulo $v\in\cC$
if and only if
\beq\label{eq:limitnilp}
\limsup_{m\to\infty}
\sup\l(1,\l|G_{[m]}\r|_{v,Gauss}\r)^{1/m}<
|[\kappa_v]_q|_v^{-1/\kappa_v}\,.
\eeq
\end{cor}

\begin{proof}
One side of the implication is an immediate consequence of the proposition above.
On the other hand, the assumption \eqref{eq:limitnilp} implies that
$$
\limsup_{m\to\infty}
\sup\l(1,\l|G_m\r|_{v,Gauss}\r)^{1/m}<1\,,
$$
which clearly implies that there exists $n$ such that
$|G_{n\kappa_v}|_{v,Gauss}\leq |\varpi_v|_v$.
\end{proof}

We finally obtain the following proposition, that will be useful
in the proof of Theorem \ref{thm:regularity}:

\begin{prop}\label{prop:globalnilpotence}
Let $\cM$ be  $q$-difference module over $K(x)$ of type $G$.
Let $\cC_0$ be the set of $v\in\cC$ such that $\cM$ does not have nilpotent reduction modulo $v$.
Then
$$
\sum_{v\in\cC_0}\frac1{\kappa_v}<+\infty\,.
$$
In particular, $\cM$ has nilpotent reduction modulo $v$ for infinitely many
$v\in\cC$.
\end{prop}

The proof relies on the following lemma:

\begin{lemma}\label{lemma:limit}
The following limit exists:
$$
\lim_{n\to\infty}\frac 1n\lgp\l(\sup_{s\leq n}\l|G_{[s]}(x)\r|_{v,Gauss}\r)\,.
$$
\end{lemma}

\begin{proof}
The proof is essentially the same as the proof of \cite[4.2.7]{DVInv}, a part from the
estimate of the $q$-factorials (\cf Lemma \ref{lemma:qfactorials} above).
The key point is the following formula:
$$
G_{[n+s]}(x)=\sum_{i+j=n}
\frac{[n]_q^![s]_q^!}{[s+n]^!}\frac{\dq^j}{[j]_q^!}\l(G_{[s]}(q^ix)\r)G_{[i]}(x)\,,
\,\forall\, s,n\in\N\,,
$$
obtained iterating the Leibniz rule.
\end{proof}

\begin{proof}[Proof of Proposition \ref{prop:globalnilpotence}]
The Fatou lemma, together with Lemma \ref{lemma:limit}, implies:
$$
\sum_{v\in\cC}\lim_{n\to\infty}\frac 1n\lgp\l(\sup_{s\leq n}\l|G_{[s]}(x)\r|_{v,Gauss}\r)
    \leq\liminf_{n\to\infty}\frac 1n\sum_{v\in\cC}\lgp\l(\sup_{s\leq n}\l|G_{[s]}(x)\r|_{v,Gauss}\r)
\leq\sg^{(q)}_\cC(\cM)<\infty\,.
$$
It follows from Corollary \ref{cor:nonnilpextimate} that:
$$
\sum_{v\in\cC_0}\frac{\lgp|[\kappa_v]_q|_v^{-1}}{\kappa_v}<\infty
$$
and hence that
$$
\sum_{v\in\cC_0}\frac{\log d^{-1}}{\kappa_v}<\infty\,,
$$
since only a finite number of places of $K/k(q)$ are ramified.
\end{proof}

\section{Proof of Theorem \ref{thm:regularity}}
\label{sec:proofregularity}

It is enough to prove that $0$ is a regular singular point for $\cM$,
the proof at $\infty$ being completely analogous.
\par
Let $r\in\N$ be a divisor of $\nu!$ and let $L$ be a finite extension of $K$ containing an
element $\widetilde q$
such that $\widetilde q^r=q$. We consider the field extension $K(x)\hookrightarrow L(t)$, $x\mapsto t^r$.
The field $L(t)$ has a natural structure of
$\widetilde q$-difference algebra extending the $q$-difference structure of $K(x)$.
Remark that:

\begin{lemma}\label{basis change}
The $q$-difference module $\cM$ is regular singular at $x=0$
if and only if
the $\widetilde q$-difference module $\cM_{L(t)}:=(M\otimes_{K(t)} L(t),
\Sg_{\wtilde q}:=\Sgq\otimes\sg_{\wtilde q})$
is regular singular at $t=0$.
\end{lemma}

\begin{proof}
It is enough to notice that if $\ul e$ is a cyclic basis for $\cM$, then $\ul e\otimes 1$ is a cyclic basis
for $\cM_{L(t)}$ and $\Sg_{\widetilde q}(\ul e\otimes 1)=\Sgq(\ul e)\otimes 1$.
\end{proof}

The next lemma can be deduced from the formal classification of $q$-difference modules
(\cf \cite[Cor. 9 and \S9, 3)]{Praag}, \cite[Thm. 3.1.7]{sauloyfiltration}):

\begin{lemma}\label{rationalgauge}
There exist an extension $L(t)/K(x)$ as above, a basis $\ul f$ of the  $\widetilde q$-difference module $\cM_{L(t)}$,
such that $\Sg_{\widetilde q}\ul f=\ul f B(t)$, with $B(t)\in Gl_\mu(L(t))$, and
an integer $\ell$ such that
\beq\label{nonnilp}
\left\{%
\begin{array}{l}
    \ds B(t)=\frac{B_\ell}{t^\ell}+\frac{B_{\ell-1}}{t^{\ell-1}}+\dots\,, \hbox{as an element of $Gl_\mu(L((t)))$;} \\
    \hbox{$B_\ell$ is a constant non nilpotent matrix.} \\
\end{array}%
\right.
\eeq
\end{lemma}

\medskip
\begin{proof}[Proof of Theorem \ref{thm:regularity}]
Let $\cB\subset L(t)$ be a $\widetilde q$-difference algebra over the ring
of integers $\cO_L$ of $L$, of the same form
as \eqref{eq:algebraA}, containing the entries of $B(t)$.
Then there exists a
$\cB$-lattice $\cN$ of $\cM_{L(t)}$ inheriting the $\widetilde q$-difference
module structure from $\cM_{L(t)}$
and having the following properties:\\
1. $\cN$ has nilpotent reduction modulo infinitely many cyclotomic places of $L$;\\
2. there exists a basis $\ul f$ of $\cN$ over $\cB$ such that $\Sg_{\widetilde q}\ul f=\ul f B(t)$ and $B(t)$ verifies (\ref{nonnilp}).
\par
Iterating the operator $\Sg_{\widetilde q}$ we obtain:
$$
\Sg_{\widetilde q}^m(\ul f)
=\ul fB(t)B({\widetilde q}t)\cdots B({\widetilde q}^{m-1}t)
=\ul f\l({B_\ell^m\over\wtilde q^{\ell m(\ell m-1)\over 2}x^{m\ell}}+h.o.t.\r)
$$
We know that for infinitely many cyclotomic places $w$ of $L$,
the matrix $B(t)$ verifies
\beq\label{uniprel}
\l(B(t)B({\widetilde q}t)\cdots B({\widetilde q}^{\kappa_w-1}t)
-1\r)^{n(w)}\equiv 0\hbox{\ mod $\varpi_w$}\,,
\eeq
where $\varpi_w$ is an uniformizer of the place $w$,
$\kappa_w$ is the order $\widetilde q$ modulo $\varpi_w$
and $n(w)$ is a convenient positive integer.
Suppose that $\ell\neq 0$. Then $B_\ell^{\kappa_w}\equiv 0$ modulo $\varpi_w$,
for infinitely many $w$, and hence
$B_\ell$ is a nilpotent matrix, in contradiction with lemma \ref{rationalgauge}.
So necessarily $\ell=0$.

Finally we have $\Sg_{\widetilde q}(\ul f)=\ul f\l(B_0+h.o.t\r)$. It follows from (\ref{nonnilp})
that $B_0$ is actually invertible, which implies that $\cM_{L(t)}$ is regular singular at $0$.
Lemma \ref{basis change} allows to conclude.
\end{proof}

\section{Proof of Theorem \ref{thm:chudnovsky}}
\label{sec:proofchudnovsky}

\subsection{Idea of the proof.}

The hypothesis states that there exists a vector
$\vec y={}^t(y_0,\dots,y_{\nu-1})\in K[[x]]^\nu$, which is
solution of the $q$-difference system:
\beq\label{eq:sys}
\vec y(qx)=A_1(x)\vec y(x)\,,
\eeq
and therefore
of the systems
$\dq^n\vec y=G_n(x)\vec y$ and $\sgq^n\vec y=A_n(x)\vec y$ for any $n\geq 1$,
having the property that $y_0,\dots,y_{\nu-1}$ are linearly independent over $K(x)$.
We recall that
$$
G_{n+1}(x)=G_n(qx)G_1(x)+\dq G_n(x)
$$
and that
$$
A_{n+1}(x)=A_n(qx)A_1(x)\,.
$$

\par
Let us consider the operator:
$$
\La=A_1(x)^{-1}\circ\l(\dq-G_1(x)\r)\,.
$$
We know that there exists an extension $\cU$ of $K(x)$ (for instance the universal
Picard-Vessiot ring constructed in \cite[\S12.1]{vdPutSingerDifference})
such that we can find
an invertible matrix $\cY$ with coefficient in $\cU$
solution of our system $\dq\cY=G_1\cY$.
An explicit calculation shows that:
$$
\dq\circ\cY^{-1}=\l(\sgq\cY\r)^{-1}\l(\dq-G_1(x)\r)=\cY^{-1}A_1(x)^{-1}\l(\dq-G_1(x)\r)
$$
and therefore that:
\beq\label{eq:LambdaSolution}
\La^n=\cY\circ\dq^n\circ\cY^{-1}\,,\
\hbox{for all integers $n\geq 0$.}
\eeq
We set ${n\choose i}_q=\frac{[n]_q^!}{[i]_q^![n-1]_q^!}$, for any pair of integers
$n\geq i\geq 0$.
The twisted $q$-binomial formula shows that $\l|{n\choose i}_q\r|_v\leq 1$
for any $v\in\cP_f$.
\par
The proof of Theorem \ref{thm:chudnovsky} is based on the following $q$-analogue of \cite[VI, \S1]{AGfunctions}:

\begin{prop}\label{prop:Gn}
There exist $\a_0^{(n)},\dots,\a_n^{(n)}\in K$
such that for all $\vec P\in K[x]^\nu$ and all $n\geq 0$
we have:
\beq
G_{[n]}\vec P=
\sum_{i=0}^n \frac{(-1)^i}{[n]_q^!}{n\choose i}_q\a_i^{(n)}
\dq^{n-i}\circ A_i(x)\La^i(\vec P)\,,
\eeq
with $|\a_i{(n)}|_v\leq 1$, for any $v\in\cP_f$ and  $n\geq i\geq 0$.
\end{prop}

\begin{proof}
The iterated twisted Leibniz Formula (\cf for instance \cite[1.1.8.1]{DVInv})
$$
\dq^n(fg)=\sum_{j=0}^n{n\choose j}_q\sgq^j\l(\dq^{n-j}(f)\r)\dq^j(g)\,,
\,\forall f,g\in\cU\,
$$
implies
$$
\begin{array}{rcl}
\lefteqn{\ds\sum_{i=0}^n{(-1)^i\over[n]_q^!}
    {n\choose i}_q\a_i^{(n)}
    \dq^{n-i}\circ A_i(x)\circ\La^i(\vec P)}\\ \\
&=&\ds\sum_{i=0}^n {(-1)^i\over[n]_q^!}{n\choose i}_q\a_i^{(n)}
    \dq^{n-i}\circ\sgq^i(\cY)\circ
    \dq^i\circ\cY^{-1}(\vec P)\\ \\
&=&\ds\sum_{i=0}^n {(-1)^i\over[n]_q^!}{n\choose i}_q\a_i^{(n)}
    \sum_{j=0}^{n-i}{n-i\choose j}_qq^{ij}\sgq^{n-j}\l(\dq^j(\cY)\r)
    \circ\dq^{n-j}\circ\cY^{-1}(\vec P)\\ \\
&=&\ds\sum_{j=0}^n\l(\sum_{i=0}^{n-j} {(-1)^i\over[n]_q^!}{n\choose i}_q
    {n-i\choose j}_qq^{ij}\a_i^{(n)}\r)\sgq^{n-j}\l(\dq^j(\cY)\r)
    \circ\dq^{n-j}\circ\cY^{-1}(\vec P)\\ \\
&=&\ds\sum_{j=0}^n{1\over[n-j]_q^![j]_q^!}
    \l(\sum_{i=0}^{n-j}(-1)^i{n-j\choose i}_q
    q^{ij}\a_i^{(n)}\r)\sgq^{n-j}\l(\dq^j(\cY)\r)
    \circ\dq^{n-j}\circ\cY^{-1}(\vec P)\,.
\end{array}
$$
We have to solve the linear system:
$$
\sum_{i=0}^{n-j}(-1)^i{n-j\choose i}_q
q^{ij}\a_i^{(n)}=
\begin{cases}
1&\hbox{if $n=j$,} \\
0&\hbox{otherwise.}
\end{cases}
$$
For $n=j$ we obtain $\a_0^{(n)}=1$. We suppose that we have already determined
$\a_0^{(n)},\dots,\a_{k-1}^{(n)}$.
For $n-j=k$ we get:
$$
\sum_{i=0}^{k-1}(-1)^i{k\choose i}_q
q^{i(n-k)}\a_i^{(n)}=(-1)^{k+1}\a_k^{(n)}q^{k(n-k)}\,.
$$
This proves also that $|\a_k^{(n)}|_v\leq 1$ for ant $v\in\cP_f$.
\end{proof}

For all $\vec P={}^t(P_0,\dots,P_\nu-1)\in K[x]^\nu$ and $n\geq 0$ we set:
$$
\vec R_n={\La^n\over[n]_q^!}(\vec P)
$$
and:
$$
R^{<n>}=
\begin{pmatrix}
{n\choose n}_q\vec R_n & {n+1\choose n}_q\vec R_{n+1}
\dots & {n+\nu-1\choose n}_q\vec R_{n+\nu-1}
\end{pmatrix}\,.
$$
Therefore we obtain the identity:

\begin{cor}\label{cor:Rzero}
$$
G_{[n]}R^{<0>}=
\sum_{i=0}^n(-1)^i\a_i^{(n)}{\dq^{n-i}\over [n-i]_q^!}
\circ A_i(x)R^{<i>}
$$
\end{cor}

\begin{rmk}\label{rmk:programma}
In order to obtain an estimate of $\sg^{(q)}_{\cP_f}(\cM)$ we want to estimate
the matrices $G_{[n]}(x)$. The main point of the proof is the construction
of a vector $\vec P$, linked to the solution vector
$\vec y$ of \eqref{eq:sys}, such that $R^{<0>}$ is an invertible matrix.
\par
The proof is divided in step:
in step 1 we construct $\vec P$; in step 2 we prove that $R^{<0>}$ is invertible;
step 3 and 4 are devoted to the estimate of $G_{[n]}(x)$ and of $\sg^{(q)}_{\cP_f}(\cM)$.
\end{rmk}

\subsection{Step 1. Hermite-Pad\'e approximations of $\vec y$.}

We denote by $\deg$ the usual degree in $x$  and by $\ord$ the order at $x=0$.
We extend their definitions to vectors as follows:
$$
\begin{array}{l}
\ds\deg\vec P(x)=\sup_{i=0,\dots,\nu-1}\deg P_i(x)\,,
  \hbox{\ for all $\vec P=
  {}^t\l(P_0(x),\dots,P_{\nu-1}(x)\r)\in K[x]^\nu$.}\\ \\
\ds\ord\vec P(x)=\inf_{i=0,\dots,\nu-1}\ord P_i(x)\,,
  \hbox{\ for all $\vec P=
  {}^t\l(P_0(x),\dots,P_{\nu-1}(x)\r)\in K((x))^\nu$.}
\end{array}
$$
Moreover we set:
$$
\begin{cases}
 \l(\sum_{n\geq 0}\vec a_n x^n\r)_{\leq N}=
    \sum_{n\leq N}\vec a_n x^n\,,\\ \\
 \l(\sum_{n\geq 0}\vec a_n x^n\r)_{> N}=
 \sum_{n>N}\vec a_n x^n\,,
\end{cases}
\hskip 15pt \hbox{for all }
\sum_{n\geq 0}\vec a_n x^n\in K[[x]]^\nu\,.
$$
Finally, for $g(x)=\sum_{n\geq 0}g_n x^n\in K[x]$ and
for $\vec y=\sum_{n\geq 0}\vec y_n x^n\in K[[x]]^\nu$
we set:
$$
h(g,v)=\sup_{n}\lgp |g_n|_v\,,\ \forall\ v\in\cP\,,
$$
$$
h(g)=\sum_{v\in\cP}h(g,v)
$$
and
$$
\widetilde h(n,v)=\sup_{s\leq n}\lgp |\vec y_s|_v\,,
\ \forall\ v\in\cP\,,
$$
where $|\vec y_s|_v$ is the maximum of the $v$-adic absolute value of the
entries of $\vec y_s$.

The following lemma is proved in
\cite[VI, \S3]{AGfunctions} or \cite[Chap. VIII,\S3]{DGS} in the case of a number field.
The proof in the present case is exactly the same, apart from the fact that there are
no archimedean places in $\cP$:

\begin{prop}\label{prop:costruzioneg}
Let $\tau\in (0,1)$ be a constant and
$\vec y=\sum_{n\geq 0}\vec y_n x^n\in K[[x]]^\nu$.
For all integers $N>0$ there exists $\vec g(x)\in K[x]^\nu$
having the following properties:
\beq
\deg g(x)\leq N\,;
\eeq
\beq
\ord(g\vec y)_{\leq N}\geq 1+N+\l[N{1-\tau\over\nu}\r]\,;
\eeq
\beq
h(g)\leq const+{1-\tau\over \tau}\sum_{v\in\cP}\widetilde h\l(N+\l[N{1-\tau\over\nu}\r],v\r)\,.
\eeq
\end{prop}

From now on we will assume that
$\vec P(x)=(g\vec y)_{\leq N}$.

\begin{prop}\label{prop:formulatroncata}
Let $Q_1(x)\in{\mathcal V}_K[x]$ be a polynomial such that
$Q_1(x)A_1^{-1}(x)\in M_{\nu\times\nu}(K[x])$.
We set:
$$
Q_0=1\hbox{\ and\ }
Q_{n}(x)=Q_1(x)Q_{n-1}(qx)\,,\
\hbox{for all $n\geq 1$,}
$$
and
$$
t=\sup\l(\deg(Q_1(x)A_1^{-1}(x)),\deg Q_1(x)\r)\,.
$$
If $n\leq{N\over t}{1-\tau\over\nu}$, then
$$
\l(x^nQ_n(x){\dq^ng\over[n]_q^!}(x)\vec y(x)\r)_{\leq N+nt}=
x^nQ_n(x)\vec R_n\,.
$$
\end{prop}

The proposition above is a consequence of the following lemmas:

\begin{lemma}\label{lemma:gradoR}
For each $n\geq 0$ we have:
\beq
x^nQ_n(x)\vec R_n(x)\in K[x]^\nu\,;
\eeq
\beq
\deg x^nQ_n(x)\vec R_n(x)\leq N+nt\,.
\eeq
\end{lemma}

\begin{proof}
Clearly $\vec R_0=(g\vec y)_{\leq N}\in K[x]^\nu$.
We recall that there exist $c_{i,n}\in K$ such that
(\cf \cite[1.1.10]{DVInv}):
$$
\dq^n=\frac{(-1)^n}{(q-1)^nx^n}(\sgq-1)(\sgq-q)\cdots(\sgq-q^{n-1})
=\frac{(-1)^n}{(q-1)^nx^n}\sum_{i=1}^n c_{i,n}\sgq^i\,,
$$
for each $n\geq 1$. Therefore we obtain:
$$
\begin{array}{rcl}
x^nQ_n(x)\vec R_n
&=&\ds x^n Q_n(x)\cY{\dq^n\over[n]_q^!}\l(\cY^{-1}\vec P\r)\\
&=&\ds {Q_n(x)\cY\over[n]_q^!(q-1)^n}
    \sum_{i=0}^nc_{i,n}\sgq^i\l(\cY^{-1}\vec P\r)\\
&=&\ds {1\over[n]_q^!(q-1)^n}
    \sum_{i=0}^nc_{i,n}Q_n(x)A_i^{-1}(x)\sgq^i(\vec P)\,.
\end{array}
$$
Since $A_i(x)=A_1(q^{i-1}x)\cdots A_1(x)$, we conclude that $x^n Q_n(x)\vec R_n\in K[x]^\nu$ and:
$$
\begin{array}{rcl}
\deg x^n Q_n(x)\vec R_n
&\leq&\ds\sup_{i=0,\dots,n}\deg\l(Q_n(x)A_i^{-1}(x)\sgq^i(\vec P)\r)\\
&\leq&\ds\sup_{i=0,\dots,n}\l(\deg(Q_i(x)A_i^{-1}(x))+\deg Q_{n-i}(q^i x)
    +\deg\sgq^i(\vec P)\r)\\
&\leq& N+nt\,.
\end{array}
$$
\end{proof}

\begin{lemma}
$$
\ord\l(x^nQ_n(x)
{\dq^n(g)\over[n]_q^!}(x)\vec y(x)
-x^nQ_n(x)\vec R_n\r)\geq 1+N+\l[N {1-\tau\over\nu}\r]\,.
$$
\end{lemma}

\begin{proof}
We have:
$$
\begin{array}{rcl}
\lefteqn{x^nQ_n(x){\dq^n(g)\over[n]_q^!}(x)\vec y(x)
    -x^nQ_n(x)\vec R_n}\\
&=&\ds {1\over[n]_q^!(q-1)^n}
    \sum_{l=0}^n c_{l,n}Q_n(x)\l(\sgq^l(g(x))\vec y(x)
    -\cY\sgq^l\l(\cY^{-1}\vec P\r)\r)\\
&=&\ds {1\over[n]_q^!(q-1)^n}
    \sum_{l=0}^n c_{l,n}Q_n(x)\l(\sgq^l(g(x))\vec y(x)
    -A_l^{-1}(x)\sgq^l(\vec P)\r)\,.
\end{array}
$$
Let $\vec H_{l}=Q_l(x)\sgq^l(g(x))\vec y(x)
-Q_l(x)A_l^{-1}(x)\sgq^l(\vec P)$.
Since:
$$
A_1^{-1}(x)Q_1(x)\sgq\l(\vec H_l\r)=\vec H_{l+1}\,,
$$
by induction on $l$ we obtain:
$$
\ord \vec H_l\geq\ord \vec H_{l-1}\geq
\ord\l(g(x)\vec y(x)-\vec P(x)\r)\geq 1+N+\l[N{1-\tau\over\nu}\r]\,.
$$
\end{proof}

\subsection{Step 2. The matrix $R^{<0>}$.}

\begin{thm}\label{thm:determinante}
Let
$\vec y(x)={}^t\l(y_0(x),\dots,y_{\nu-1}(x)\r)\in K[[x]]^\nu$
a solution vector of $\La Y=0$, such that
$y_0(x),\dots,y_{\nu-1}(x)$ are linearly independent over
$K(x)$.
Then there exists a constant $C(\La)$, depending only on $\La$, such that if
$$
\vec P={}^t(P_0\dots,P_{\nu-1})\in K[x]^\nu\smallsetminus \{\ul 0\}
$$
has the following property:
\beq\label{eq:determinante}
\ord \det
\begin{pmatrix}
P_i& P_j\\ y_i & y_j
\end{pmatrix}\geq
\deg\vec P(x)+C(\La)\,,
\forall i, j=0,\dots,\nu-1,
\eeq
then the matrix $R^{<0>}$ is invertible.
\end{thm}

\begin{rmk}\label{rmk:ossP}
We remark that if we choose $g$ as in Propositions \ref{prop:costruzioneg}
and \ref{prop:formulatroncata} and $\vec P=(g\vec y)_{\leq N}$,
for $N>>0$ we have:
$$
N{1-\tau\over\nu}\geq C(\La)\,.
$$
Therefore the condition \eqref{eq:determinante} is satisfied since:
$$
\ord \det
\begin{pmatrix}
P_i & P_j\\ y_i& y_j
\end{pmatrix}=
\ord \det\begin{pmatrix}\l(gy_i\r)_{>N} & \l(gy_j\r)_{>N}\\ y_i& y_j\end{pmatrix}
\geq 1+N+N{1-\tau\over\nu}\,.
$$
\end{rmk}

We recall the Shidlovsky's Lemma that we will need on the proof of Theorem \ref{thm:determinante}.

\begin{defn}
We define total degree of ${f(x)\over g(x)}\in K(x)$
as:
$$
{\rm deg.tot}{f(x)\over g(x)}=\deg f(x)+\deg g(x)\,.
$$
\end{defn}

\begin{lemma}[Shidlovsky's Lemma; \cf for instance {\cite[Chap. VIII, 2.2]{DGS}}]
Let ${\mathcal G}/K(x)$ be a field extension and let $V\subset{\mathcal G}$
a $K$-vector space of finite dimension.
Then the total degree of the elements of $K(x)$
that can be written as quotient of two element of $V$ is bounded.
\end{lemma}

\begin{proof}
[Proof of the Theorem \ref{thm:determinante}]
Let $\cY$ be an invertible matrix with coefficients in an extension $\cU$ of
$K(x)$ such that $\La\cY=0$ and let $C$ be the field of constant of
$\cU$ with respect to $\dq$.
The matrix
$$
R^{<0>}=\cY\l(\cY^{-1}\vec P, \dq\l(\cY^{-1}\vec P\r), \cdots ,
{\dq^{\nu-1}\over [\nu-1]_q^!}\l(\cY^{-1}\vec P\r)\r)
$$
is invertible if and only if
$$
{\rm rank}\l(\cY^{-1}R^{<0>}\r)={\rm rank}
\l(\cY^{-1}\vec P, \sgq\l(\cY^{-1}\vec P\r),
\dots, \sgq^{\nu-1}\l(\cY^{-1}\vec P\r)\r)
$$
is maximal.
Let us suppose that
$$
{\rm rank}\l(\cY^{-1}R^{<0>}\r)=r<\nu\,.
$$
Then the $q$-analogue of the wronskian lemma (\cf for instance \cite[\S1.2]{DVInv})
implies that there exists an invertible matrix
$M$ with coefficients in $C$ such that the first column of $M\cY^{-1}R^{<0>}$ is equal to:
$$
M\cY^{-1}\vec P=
{}^t\l(\widetilde w_0,\widetilde w_1,\dots,\widetilde w_{r-1},
0,\dots,0\r)\,.
$$
The matrix $\cY M^{-1}$ still verifies the $q$-difference
equation $\La Y=0$, so we will write
$\cY$ instead of $\cY M^{-1}$, to simplify notation.
We set:
$$
\vec S_n=\cY\circ\sgq^n\circ\cY^{-1}\vec P\,,\
\forall n\geq 0\,,
$$
$$
S^{<0>}=\l(\vec S_0 ,\dots,\vec S_{\nu-1}\r)
=\begin{pmatrix}
S_{IJ} & S_{IJ^\p}\\ S_{I^\p J}& S_{I^\p J^\p}
\end{pmatrix}
$$
and
$$
\cY^{-1}=\begin{pmatrix}\cY_{JL} & \cY_{JL^\p}\\
             \cY_{J^\p L}& \cY_{J^\p L^\p}
             \end{pmatrix}\,,
$$
where $I=J=L=\{0,1,\dots,r-1\}$ and $I^\p=J^\p=L^\p=\{r,\dots,\nu-1\}$.
We have:
$$
\begin{pmatrix}
\cY_{JL} & \cY_{JL^\p}\\
             \cY_{J^\p L}& \cY_{J^\p L^\p}
\end{pmatrix}
\begin{pmatrix}
S_{IJ} & S_{IJ^\p}\\ S_{I^\p J}& S_{I^\p J^\p}
\end{pmatrix}
=\begin{pmatrix}
\sgq^i\l(\cY^{-1}\vec P\r)
\end{pmatrix}_{i=0,\dots,\nu-1}
=\begin{pmatrix}
A\\ 0
\end{pmatrix}\,,
$$
with $A\in M_{r\times\nu}(K(x))$, and therefore:
$$
\cY_{J^\p L}S_{IJ}+\cY_{J^\p L^\p}S_{I^\p J}=0\,.
$$
Because of our choice of $\cY$,
the vectors $\vec S_0,\dots,\vec S_{r-1}$ are linearly independent
over $K(x)$,
so by permutation of the entries of the vector $\vec P$ we can suppose that
the matrix $S_{IJ}$ is invertible.
\par
Let $B=S_{I^\p J}S_{IJ}^{-1}$.
Since $S^{<0>}\in M_{\nu\times\nu}\l(K(x)\r)$ is independent of the
choice of the matrix $\cY$, the same is true for $B$.
The matrix $\cY$ is invertible and
$$
\begin{pmatrix}
\cY_{J^\p L}& \cY_{J^\p L^\p}
\end{pmatrix}=
\cY_{J^\p L^\p}
\begin{pmatrix}
-B& I_{\nu-r}
\end{pmatrix}\,,
$$
therefore the matrix $\cY_{J^\p L^\p}$ is also invertible
and we have:
$$
B=-\cY_{J^\p L^\p}^{-1}\cY_{J^\p L}\,.
$$
The coefficients of the matrix $B$ can be written in the form
$\xi/\eta$, where $\xi$ and $\eta$ are elements of the
$K$-vector space of polynomials of degree less or equal to
$\nu-r$ with coefficients in $K$ in the entries of the matrix $\cY$.
By Shidlovsky's lemma the total degree of the entries of the matrix
$B$ is bounded by a constant depending only on the
$q$-difference system $\La$.
\par
Let us consider the matrices:
$$
Q_1=\begin{pmatrix}
y_{\nu-1}&0&0&\cdots&0\\
y_1&-y_0&0&\cdots&0\\
y_2&0&-y_0&\cdots&0\\
\vdots&\vdots&\vdots&\ddots&\vdots\\
y_{r-1}&0&0&\cdots&-y_0
\end{pmatrix}
\in M_{r\times r}(K[[x]])
$$
and
$$
Q_2=\begin{pmatrix}
0&\cdots&0&-y_0\\
\vdots&\ddots&\vdots&\vdots\\
0&\cdots&0&0\\
0&\cdots&0&0\end{pmatrix}
\in M_{r\times\nu-r}(K[[x]])\,;
$$
we set:
$$
T=\begin{pmatrix}Q_1& Q_2\end{pmatrix}
\begin{pmatrix}S_{IJ}\\ S_{I^\p J}\end{pmatrix}
=\begin{pmatrix}Q_1& Q_2\end{pmatrix}
\begin{pmatrix}{\mathbb I}_r\\ B\end{pmatrix}
S_{IJ}\,.
$$
Let $(b_0,\dots,b_{r-1})$ be the last row of $B$.
We have:
$$
\begin{array}{rcl}
\det\l(TS_{IJ}^{-1}\r)
&=&\ds\det\l(Q_1+Q_2B\r)\\
&=&\ds\det\begin{pmatrix}
        y_{\nu-1}-y_0b_0&-y_0b_1 &-y_0b_2 &\cdots&-y_0b_{r-1}\\
        y_1             &-y_0    &0       &\cdots&0\\
        y_2             &0       &-y_0    &\cdots&0\\
        \vdots&\vdots   &\vdots  &\ddots  &\vdots \\
        y_{r-1}       &0       &0       &\cdots&-y_0
        \end{pmatrix}\\\\
&=&\ds\det\begin{pmatrix}
        y_{\nu-1}-y_0b_0-y_1b_1-\cdots-y_{r-1}b_{r-1}&0&0&\cdots&0\\
        y_1&-y_0&0&\cdots&0\\
        y_2&0&-y_0&\cdots&0\\
        \vdots&\vdots&\vdots&\ddots&\vdots\\
        y_{r-1}&0&0&\cdots&-y_0
        \end{pmatrix}\\\\
&=&\ds\l(-y_0\r)^{r-1}
        \l(y_{\nu-1}-y_0b_0-y_1b_1-\cdots-y_{r-1}b_{r-1}\r)\,.
\end{array}
$$
We notice that $\det\l(TS_{IJ}^{-1}\r)\neq 0$, since by hypothesis
$y_0,\dots,y_{\nu-1}$ are linearly independent over $K(x)$.
Our purpose is to find a lower and an upper bound  for $\ord\det\l(TS_{IJ}^{-1}\r)$.
\par
Since the total degree of the entries of $B$
is bounded by a constant depending only on $\La$, there exists
a constant $C_1$, depending on $\La$ and not on $\vec P$, such that:
$$
\ord \det\l(TS_{IJ}^{-1}\r)\leq C_1\,.
$$
Now we are going to determine a lower bound.
Let:
$$
\vec S_n={}^t\begin{pmatrix}S_{n,0},S_{n,2}, \dots , S_{n,\nu-1}\end{pmatrix}\,,
\hbox{\ pour tout $n\geq 0$;}
$$
then we have:
$$
S^{<0>}=\l(S_{i,j}\r)_{i,j\in\{0,1,\dots,\nu-1\}}\,;
$$
moreover we set:
$$
A_1^{-1}=\l( A_{i,j}\r)_{i,j\in\{0,1,\dots,\nu-1\}}\,.
$$
The elements of the first row of $T$ are of the form:
$$
\det\begin{pmatrix}y_{\nu-1} & S_{s,\nu-1}\\ y_0 & S_{s,0}\end{pmatrix}\,,
\hbox{\ pour $s=0,\dots,r-1$,}
$$
and the ones of the $i$-th row, for $i=1,\dots,r-1$:
$$
\det\begin{pmatrix}y_i & S_{s,i}\\ y_0 & S_{s,0}\end{pmatrix}\,,
\hbox{\ pour $s=0,\dots,r-1$.}
$$
Since $\vec S_{n+1}=A_1(x)^{-1}\sgq(\vec S_n)$ we have:
$$
\det\begin{pmatrix}y_i & S_{s+1,i}\\ y_j & S_{s+1,j}\end{pmatrix}=
\det\begin{pmatrix}y_i &\sum_{l} A_{i,l}\sgq(S_{s,l})\\
               y_j &\sum_{l} A_{j,l}\sgq(S_{s,l})\end{pmatrix}\,,
$$
therefore:
$$
\inf_{i,j=0,\dots,\nu-1}\ord
\det\begin{pmatrix}y_i & S_{s+1,i}\\ y_j & S_{s+1,j}\end{pmatrix}
\geq(s+1)\ord  A_1(x)^{-1}+
\inf_{i,j=0,\dots,\nu-1}\ord \det\begin{pmatrix}y_i &P_i\\ y_j &P_j\end{pmatrix}\,.
$$
Finally,
$$
\ord\det T\geq r(\nu-1)\ord  A_1(x)^{-1}+
r\inf_{i,j=0,\dots,\nu-1}\ord\det\begin{pmatrix}y_i &P_i\\ y_j &P_j\end{pmatrix}\,.
$$
By Lemma \ref{lemma:gradoR}, we obtain:
$$
\begin{array}{rcl}
\ord \det S_{I,J}
&\leq&\ds\deg\hbox{(numerator of $\det S_{I,J}$)}\\
&\leq&\ds\sum_{i=0}^{r-1}\deg\hbox{(numerator of $\vec S_i$)}\\
&\leq&\ds r\deg\vec P+t{r(r-1)\over 2}\,.
\end{array}
$$
We deduce that:
$$
\begin{array}{rcl}
\ord \det\l(TS_{I,J}^{-1}\r)
&\geq&\ds\ord \det\l(T\r)-\ord \det\l(S_{I,J}\r)\\
&\geq&\ds r\l((\nu-1)\ord  A_1(x)^{-1}+
    \inf_{i,j=0,\dots,\nu-1}\ord\det\begin{pmatrix}y_i &P_i\\ y_j &P_j\end{pmatrix}
    -\deg\vec P-t{(r-1)\over 2}\r)\\
&\geq&\ds r\l(\inf_{i,j}\ord \det\begin{pmatrix}y_i &P_i\\ y_j &P_j\end{pmatrix}
      -\deg  \vec P\r)+C_2\,,
\end{array}$$
where $C_2$ is a constant depending only on $\La$.
To conclude it is enough to choose a constant $C(\La)>{C_1-C_2\over r}$.
\end{proof}

\subsection{Step 3. First part of estimates.}

We set:
$$
\begin{array}{l}
\ds y=\sum_{n\geq 0}\vec y_n x^n\,,
  \hbox{\ with $\vec y_n\in K^\nu$}\,,\\
\ds \sg_f\l(\vec y\r)=\limsup_{n\rightarrow+\infty}{1\over n}
    \l(\sum_{v\in\cP_f}\sup_{s\leq n}
    \lgp\l |\vec y_s\r|_{v}\r)\,,\\
\ds \sg_\infty\l(\vec y\r)=\limsup_{n\rightarrow+\infty}{1\over n}
    \l(\sum_{v\in\cP_\infty}\sup_{s\leq n}
    \lgp\l |\vec y_s\r|_{v}\r)\,.
\end{array}
$$
We recall that we are working under the assumption:
$$
\sg(y)=\limsup_{n\rightarrow+\infty}{1\over n}
\l(\sum_{v\in\cP}\widetilde h(n,v)\r)
=\sg_f(\vec y)+\sg_\infty(\vec y)<+\infty
$$
and that we want to show that $\sg^{(q)}_{\cC}(\cM)\leq\infty$. Since
$\sg^{(q)}_{\cC}(\cM)\leq\sg^{(q)}_{\cP_f}(\cM)$, we will rather show that:
$$
\sg^{(q)}_{\cP_f}(\cM)=\limsup_{n\rightarrow+\infty}{1\over n}
\l(\sum_{v\in\cP_f}h(\cM,n,v)\r)< \infty\,,
$$
where:
$$
h(\cM,n,v)=\sup_{s\leq n}
\lgp\l|{G_n\over [n]_q^!}\r|_{v,Gauss}\,.
$$
\par
In the sequel $g$ will be a polynomial constructed as in
Proposition \ref{prop:costruzioneg}.  For such a choice of $g$ and for
$\vec P=(g\vec y)_{\leq N}$,  the hypothesis of Corollary \ref{cor:Rzero},
Proposition \ref{prop:costruzioneg} and Theorem \ref{thm:determinante}
are satisfied.

\begin{prop}\label{prop:primestime}
We have:
$$
\sg^{(q)}_{\cP_f}(\cM)\leq\sg_f(\vec y)\l({\nu^2 t\over 1-\tau}+t\r)+\Omega
+\sum_{v\in\cP_f}\lgp|A_1(x)|_{v,Gauss}\,,
$$
where:
$$
\Omega=\limsup_{n\rightarrow+\infty}{1\over n}
\l(\nu\sum_{v\in\cP_f}h(g,v)+
\sum_{v\in\cP_f}\lgp\l|\l(\prod_{i=1}^{\nu-1}Q_i(x)\r)
\Delta(x)\r|^{-1}_{v,Gauss}\r)\,.
$$
\end{prop}

\begin{proof}
We fix $N,n>>0$ such that:
\beq\label{eq:sceltaN}
n+\nu-1\leq{N\over t}{1-\tau\over\nu}\,.
\eeq
Proposition \ref{prop:formulatroncata} and Corollary \ref{cor:Rzero}
implies that for all integers $s\leq n+\nu-1$, we have:
\beq\label{eq:condsceltaN}
\l(x^sQ_s(x)
{\dq^s g\over [s]_q^!}(x)\vec y(x)\r)_{\leq N+st}=
x^sQ_s(x)\vec R_s\,
\eeq
and:
$$
G_[s]=\sum_{i\leq s}
(-1)^i\a_i^{(n)}{\dq^{s-i}\over[s-i]_q^!}(A_i(x)R^{<i>})
\l(R^{<0>}\r)^{-1}\,.
$$
For all $v\in\cP_f$ we deduce:
$$
\begin{array}{rcl}
\l |G_[s]\r|_{v,Gauss}
&\leq&\ds \l(\sup_{i\leq s}\l|{\dq^{s-i}\over[s-i]_q^!}(A_i(x)R^{<i>})\r|_{v,Gauss}\r)
         \l|\adj R^{<0>}\r|_{v,Gauss}\l|\det R^{<0>}\r|^{-1}_{v,Gauss}\\
&\leq&\ds \l(\sup_{i\leq s}\l|A_i(x)R^{<i>}\r|_{v,Gauss}\r)
         \l|\adj R^{<0>}\r|_{v,Gauss}\l|\det R^{<0>}\r|^{-1}_{v,Gauss}\\
&\leq&\ds  C_{1,v}^s\l(\sup_{i\leq s+\nu-1}|\vec R_i|_{v,Gauss}\r)
         \l(\sup_{i\leq \nu-1}|\vec R_i|_{v,Gauss}\r)^{\nu-1}
        |\Delta(x)|_{v,Gauss}^{-1}\,,
\end{array}
$$
where we have set:
$$
C_{1,v}=\sup(1,|A_1(x)|_{v,Gauss})
$$
and
$$
\Delta(x)=\det R^{<0>}(x)\,.
$$
Taking into account our choice of $N$ and $n$ and \eqref{eq:condsceltaN},
for all $i\leq n+\nu-1$ we have:
$$
\begin{array}{rcl}
\l|\vec R_i\r|_{v,Gauss}
&\leq&\ds \l|Q_i(x)\r|^{-1}_{v,Gauss}\l|Q_i(x)\r|_{v,Gauss}\l|g\r|_{v,Gauss}
          \l|\l(\vec y\r)_{\leq N+it}\r|_{v,Gauss}\\
&\leq&\ds \l|g\r|_{v,Gauss}\l|\l(\vec y\r)_{\leq N+it}\r|_{v,Gauss}\,,
\end{array}
$$
therefore:
$$
\begin{array}{rcl}
\sup_{s\leq n}\lgp\l |G_[s]\r|_{v,Gauss}
&\leq&\ds n\log C_{1,v}+\widetilde h\l(N+(n+\nu-1)t,v\r)\\
&&\ds +(\nu-1)\widetilde h\l(N+(\nu-1)t,v\r)
         +\nu h(g,v)+\lgp\l|\Delta\r|^{-1}_{v,Gauss}\,.
\end{array}
$$
We set:
$$
\begin{array}{rcl}
\ol\Delta(x)
&=&\ds \vec R_0\wedge xQ_1(x)\vec R_1\wedge
          \cdots\wedge x^{\nu-1}Q_{\nu-1}(x)\vec R_{\nu-1}\\
&=&\ds x^{\nu\choose 2}\l(\prod_{i=1}^{\nu-1}Q_i(x)\r)\Delta(x)\,.
\end{array}
$$
The fact that $\l|Q_1(x)\r|_{v,Gauss}\leq 1$ and
$x^nQ^n(x)\vec R_n\in K[x]^\nu$, for all integers $n\geq 1$,
implies that $\l|\ol\Delta(x)\r|_{v,Gauss}\leq \l|\Delta(x)\r|_{v,Gauss}$,
with $\ol\Delta(x)\in K[x] $, and:
$$
\begin{array}{rcl}
\ds\sup_{s\leq n}\lgp\l |G_[s]\r|_{v,Gauss}
&\leq&\ds n\log C_{1,v}+\widetilde h\l(N+(n+\nu-1)t,v\r)\\
&&\ds +(\nu-1)\widetilde h\l(N+(\nu-1)t,v\r)
         +\nu h(g,v)+\lgp\l|\ol\Delta\r|^{-1}_{v,Gauss}\,.
\end{array}
$$
Taking into account condition \eqref{eq:sceltaN},
we fix a positive integer $k$ such that:
\beq\label{eq:sceltaNbis}
\begin{cases}
\displaystyle k>\frac{\nu(\nu-1)t}{1-\tau}\\
\displaystyle {N\over n}={\nu t\over 1-\tau}+{k-\veps_n\over n}
\hbox{\,, for some $\veps_n\in(0,1)$ fixed.}
\end{cases}
\eeq
Let us set:
$$
C_1=\sum_{v\in\cP_f}\lgp|A_1(x)|_v
$$
and
$$
\Omega=\limsup_{n\rightarrow+\infty}{1\over n}
\l(\nu\sum_{v\in\cP_f}h(g,v)+\sum_{v\in\cP_f}
\lgp \l|\ol\Delta(x)\r|^{-1}_{v,Gauss}\r)\,.
$$
We obtain:
$$
\begin{array}{rcl}
\sg^{(q)}_{\cP_f}(\cM)
&=&\ds\limsup_{n\rightarrow+\infty}{1\over n}
          \l(\sum_{v\in\cP_f\atop |1-q^\kappa|^{1/(p-1)}}\sup_{s\leq n}
          \lgp\l |{G_s\over[s]_q^!}\r|_{v,Gauss}\r)\\
&\leq&\ds\sg_f(\vec y)\limsup_{n\rightarrow+\infty}
         \l({N+(n+\nu-1)t\over n}+
         (\nu-1){N+(\nu-1)t\over n}\r)+C_1+\Omega\\
&\leq&\ds \sg_f(\vec y)\l({\nu t\over 1-\tau}+t+(\nu-1)
         {\nu t\over 1-\tau}\r)+C_1+\Omega\\
&\leq&\ds \sg_f(\vec y)\l({\nu^2 t\over 1-\tau}+t\r)+C_1+\Omega\,.
\end{array}
$$
\end{proof}

\subsection{Step 4. Conclusion of the proof of Theorem \ref{thm:chudnovsky}.}

\begin{lemma}\label{lemma:omega}
Let $\Omega$ be as in the previous proposition. Then:
$$
\Omega\leq{\nu^2 t\over 1-\tau}\sg_\infty(\vec y)+{\nu^2 t(\nu-1)\over 1-\tau}C_2+
\limsup_{n\rightarrow+\infty} {\nu\over n}h(q)\,,
$$
where
$$
C_2=\sum_{v\in\cP_\infty}\log(1+|q|_v)
$$
is a constant depending on the $v$-adic absolute value of $q$,
for all $v\in\cP_\infty$.
\end{lemma}

\begin{proof}
Let $\xi$ a root of unity such that:
$$
\ol\Delta(\xi)\neq 0\neq Q_i(\xi)\ \forall\ i=0,\dots\nu-1\,.
$$
Since $|\ol\Delta(\xi)|_v\leq|\ol\Delta(x)|_{v,Gauss}$
for all $v\in\cP_f$, the Product Formula implies that:
$$
\sum_{v\in\cP_f}\lgp \l|\ol\Delta(x)\r|^{-1}_{v,Gauss}\leq
\sum_{v\in\cP_f}\lgp \l|\ol\Delta(\xi)\r|^{-1}_v\leq
\sum_{v\in\cP_\infty}\lgp \l|\ol\Delta(\xi)\r|_v\,.
$$
We recall that:
$$
\ol\Delta(x)= \det
\begin{pmatrix}
\vec R_0 &
Q_1(x)\vec R_1 & \cdots &
Q_{\nu-1}(x)\vec R_{\nu-1}
\end{pmatrix}
$$
and that for all $s\leq \nu-1$, \eqref{eq:condsceltaN} is verified.
Moreover we have:
$$
\begin{array}{rcl}
\ds Q_s(x){\dq^s(g)\over[s]_q^!}(x)\vec y(x)
&=&\ds\sum_{n\geq 0}\l(\sum_{i+j+h=n}(Q_s)_i
  \l({\dq^s(g)\over[s]_q^!}\r)_j\vec y_h\r)x^n\\
&=&\ds\sum_{n\geq 0}\l(\sum_{i+j+h=n}(Q_s)_i
  {s+j\choose j}_q g_{s+j}\vec y_h\r)x^n\,,
\end{array}
$$
where we have used the notation:
$$
\hbox{for all $P\in K[[x]]$ and for all $n\in\N$,
$P_n$ is the coefficient of $x^n$
in $P$.}
$$
We deduce that $Q_s(\xi)\vec R_s(\xi)$
is a sum of terms of the type:
$$
(Q_s)_i{s+j\choose j}_q g_{s+j}\vec y_h\xi^n
$$
with:
$$
\begin{array}{ll}
0\leq s\leq \nu-1\,,&
0\leq i\leq\deg Q_s(x)\,,\\
0\leq j\leq N\,,\ s+j\leq N\,,&
0\leq h\leq N+(\nu-1)t\,.
\end{array}
$$
For all $v\in\cP_\infty$ we obtain:
$$
\l|Q_s(\xi)\vec R_s(\xi)\r|_v
\leq c_v\l(\sup_{s\leq j\leq N}\l|{j\choose s}_q\r|_v\r)
      \l(\sup_{h\leq N+(\nu-1)t}|\vec y_h|_v\r)
      \l(\sup_{j\leq N}|g_j|_v\r)\,,
$$
with:
$$
c_v=\sup\l(1,\sup_{s=0,\dots,\nu-1\atop i=0,\dots,\deg Q_s}
|(Q_s(x))_i|_v\r)\,.
$$
Since $|q|_v\neq 1$, for all $v\in\cP_\infty$, we have:
$$
\begin{array}{rcl}
\ds\l|{j\choose s}_q\r|_v
&=&\ds\l|{(1-q^j)\cdots(1-q^{j-s+1})\over (1-q^s)\cdots(1-q)}\r|_v\\
&\leq&\ds {(1+|q|_v^j)\cdots(1+|q|_v^{j-s+1})\over |1-|q|_v^s|_v\cdots|1-|q|_v|_v}\\
&\leq&\ds\l\{\begin{array}{ll}
   \displaystyle {(1+|q|_v)^s\over 1-|q|_v^s}
   \leq \l({1+|q|_v\over 1-|q|_v}\r)^{\nu-1}
   &\hbox{ if $|q|_v<1$; }\\ \\
   \displaystyle \l({1+|q|_v^j\over |q|_v^s-1}\r)^s
   \leq \l({1+|q|_v^N\over |q|_v^{\nu-1}-1}\r)^{\nu-1}
   &\hbox{ if $|q|_v>1$;}
   \end{array}\r.
\end{array}
$$
hence:
$$
\begin{array}{rcl}
\ds\sup_{s=0,\dots,\nu-1\atop j=s,\dots N}\l|{j\choose s}_q\r|_v
&\leq&\ds \l({\sup(1+|q|_v,1+|q|_v^N)\over
    \inf(|1-|q|_v|,|1-|q|_v^{\nu-1}|)}\r)^{\nu-1}\\
&\leq&\ds {(1+|q|_v)^{N(\nu-1)}\over\inf(|1-|q|_v|, |1-|q|_v^{\nu-1}|)^{\nu-1}}
\end{array}
$$
We obtain the following estimate:
$$
\l|Q_s(\xi)\vec R_s(\xi)\r|_v
\leq c_v{\l(1+|q|_v\r)^{N(\nu-1)}\over\inf(|1-|q|_v|,|1-|q|_v^{\nu-1}|)^{\nu-1}}
\l(\sup_{h\leq N+(\nu-1)t}|\vec y_h|_v\r)
\l(\sup_{j\leq N}|g_j|_v\r)
\,.
$$
Finally we get:
$$
\l|\ol\Delta(\xi)\r|_v\leq c_v^\nu
{\l(1+|q|_v\r)^{N(\nu-1)\nu}\over\inf(|1-|q|_v|,|1-|q|_v^{\nu-1}|)^{(\nu-1)\nu}}
\l(\sup_{h\leq N+(\nu-1)t}|\vec y_h|_v\r)^\nu
\l(\sup_{j\leq N}|g_j|_v\r)^\nu
$$
and therefore:
$$
\begin{array}{rcl}
\lefteqn{\ds\sum_{v\in\cP_\infty}\lgp\l|\ol\Delta(\xi)\r|_v
\leq\ds const+N\nu(\nu-1) C_2}\\
&&\ds-\nu(\nu-1)\sum_{v\in\cP_\infty}\log\inf(|1-|q|_v|,|1-|q|_v^{\nu-1}|)^{\nu-1}\\
&&\ds+\nu\sum_{v\in\cP_\infty}h(g,v)
    +\nu\sum_{v\in\cP_\infty}\widetilde h\l(N+(\nu-1)t,v\r)\,.
\end{array}
$$
where:
$$
C_2=\sum_{v\in\cP_\infty}\log\l(1+|q|_v\r)\,.
$$
We recall that by \eqref{eq:sceltaNbis}, we have:
$$
\lim_{n\rightarrow+\infty}{N\over n}={t\nu\over 1-\tau}
$$
and:
$$
\lim_{n\rightarrow+\infty}{\log N\over n}=0\,.
$$
So we can conclude since:
$$
\begin{array}{rcl}
\lefteqn{\ds\limsup_{n\rightarrow+\infty}{1\over n}
\sum_{v\in\cP_f}\lgp \l|\ol\Delta(x)\r|^{-1}_{v,Gauss}
    \leq \limsup_{n\rightarrow+\infty}
   {1\over n}\sum_{v\in\cP_\infty}\lgp\l|\ol\Delta(\xi)\r|_v}\\
&\leq&\ds \limsup_{n\rightarrow+\infty}
      \l(\frac{N\nu(\nu-1) C_2}{n}
      +{\nu\over n}\sum_{v\in\cP_\infty}h(g,v)
      +{\nu\over n}\sum_{v\in\cP_\infty}
      \widetilde h\l(N+(\nu-1)t,v\r)\r)\\
&\leq&\ds {t\nu^2(\nu-1)\over 1-\tau}C_2+
      \limsup_{n\rightarrow+\infty}\l({\nu\over n}\sum_{v\in\cP_\infty}h(g,v)+
      {\nu\over n}\sum_{v\in\cP_\infty}\widetilde h\l(N+(\nu-1)(t-1),v\r)\r)\\
&\leq&\ds {t\nu^2\over 1-\tau}\sg_\infty(\vec y)+{t\nu^2(\nu-1)\over 1-\tau}C_2+
      \limsup_{n\rightarrow+\infty}{\nu\over n}\sum_{v\in\cP_\infty}h(g,v)\,.
\end{array}
$$
\end{proof}

\begin{proof}[Conclusion of the proof of Theorem \ref{thm:chudnovsky}]
Proposition \ref{prop:costruzioneg} implies that:
$$
\begin{array}{rcl}
\ds\limsup_{n\rightarrow+\infty}{\nu\over n}h(g)
&\leq&\ds\limsup_{n\rightarrow+\infty}{\nu\over n}
       \l(const +{1-\tau\over\tau}
       \sum_{v\in\cP}\widetilde h\l(N+N{1-\tau\over\nu},v\r)\r)\\
&\leq&\ds\limsup_{n\rightarrow+\infty}{1-\tau\over\tau}
       {\nu\over n}\sum_{v\in\cP}\widetilde h\l(N+N{1-\tau\over\nu},v\r)\\
&\leq&\ds{1-\tau\over\tau}\nu\sg\l(\vec y\r)\limsup_{n\rightarrow+\infty}
       {1\over n}\l(N+N{1-\tau\over\nu}\r)\\
&\leq&\ds{1-\tau\over\tau}\nu\sg\l(\vec y\r)
       \l({t\nu\over 1-\tau}+t\r)\\
&\leq&\ds{1-\tau\over\tau}\nu t\l(1+{\nu\over 1-\tau}\r) \sg\l(\vec y\r)\,,
\end{array}
$$
which, combined with Propositions \ref{prop:primestime} and
\ref{lemma:omega}, implies that:
$$
\begin{array}{rcl}
\ds\sg^{(q)}_{\cP_f}(\cM)
&\leq&\ds\sg_f(\vec y)\l({\nu^2 t\over 1-\tau}+t\r)+
      \sg_\infty\l(\vec y\r){\nu^2 t\over 1-\tau}+
      \sg\l(\vec y\r){1-\tau\over\tau}\nu t\l(1+{\nu\over 1-\tau}\r)\\
&&\ds\hskip 15 pt +\log C_1+{\nu^2(\nu-1)t\over 1-\tau}C_2 \\
&\leq&\ds\sg(\vec y)\l({\nu^2 t\over 1-\tau}
      +\nu^2t\l({1\over\tau}+{1-\tau\over\nu\tau}\r)
      +t\r)
+\log C_1+{\nu^2(\nu-1)t\over 1-\tau}C_2\\
&\leq&\ds\sg(\vec y)\l(\nu^2t\l({\nu+1\over \nu}{1\over\tau}+{1\over 1-\tau}\r)
      -\nu t+t\r)
      +\log C_1+{\nu^2(\nu-1)t\over 1-\tau}C_2\,.
\end{array}
$$
The function ${\nu+1\over\nu}{1\over\tau}+{1\over 1-\tau}$
has a minimum for
$$
\tau=\l(1+\sqrt{\nu\over\nu+1}\r)^{-1}\,;
$$
for this value of $\tau$ we get:
$$
{\nu+1\over\nu}{1\over\tau}+{1\over 1-\tau}=\l(1+\sqrt{\nu+1\over\nu}\r)
\leq\begin{cases}4.95&\hbox{ for $\nu\geq 2$}\\ 5.9&\hbox{ for $\nu=1$}\end{cases}\,.
$$
Finally we have:
$$
\sg^{(q)}_{\cP_f}(\cM)\leq\log C_1+{\nu^2(\nu-1)t\over 1-\tau}C_2+
\begin{cases}\sg(\vec y)\l(4.95\nu^2 t-\nu t+(t-1)\r)&\hbox{ for $\nu\geq 2$}\\\\
       \sg(\vec y)5.9 t&\hbox{ for $\nu=1$}\end{cases}\,,
$$
where
$$
C_1=\sum_{v\in\cP_f}\lgp|A_1(x)|_{v,Gauss}
$$
and
$$
C_2=\sum_{v\in\cP_\infty}\log(1+|q|_v)\,.
$$
\end{proof}

\part{\sc\bfseries Global $q$-Gevrey series}
\section{Definition and first properties}

\emph{The notation is the same as in Part 1.}
We recall that $K$ is a finite extension of $k(q)$, equipped with its family of
ultrametric norms, normalized so that the Product Formula holds. The field
$K(x)$ is naturally a $q$-difference algebra with respect to the operator
$\sgq:f(x)\mapsto f(qx)$.

\begin{defn}
We say that the series $f(x)=\sum_{n=0}^\infty a_nx^n\in K[[x]]$ is
a \emph{global $q$-Gevrey series of orders $(s_1,s_2)\in\Q^2$}
if it is solution of a $q$-difference equation with coefficients in
$K(x)$ and
$$
\sum_{n=0}^\infty {a_n\over\l(q^{n(n-1)\over 2}\r)^{s_1}\l([n]_q^!\r)^{s_2}}x^n
$$
is a $G_q$-function.
\end{defn}

\begin{rmk}
We point out that:
\begin{enumerate}
\item
The definition above forces $s_2$ to be an integer, in fact the $q$-holonomy condition
implies that
the coefficients ${[n]_q^!}^{s_2}$, for $n\geq 1$, are all contained in a finite extension of $k(q)$.
\item
Being a global $q$-Gevrey series  of orders $(s_1,s_2)$ implies
being a $q$-Gevrey series of order
$s_1+s_2$ in the sense of \cite{BezivinBoutabaa}
for all $v\in\cP_\infty$ extending the $q^{-1}$-adic norm, \ie for the norms that verify
$|q|_v>1$: this simply means that $|q^{s_1 n(n-1)\over 2}{[n]_q^!}^{s_2}|_v$ as the
 same growth as $|q|_v^{(s_1+s_2)\frac{n(n-1)}2}$.
If $v\in\cP_\infty$ and $|q|_v<1$, then $|[n]_q|_v=1$. Therefore  a global
$q$-Gevrey series  of orders $(s_1,s_2)$ is a
$q$-Gevrey series of order $s_1$ in the sense of \cite{BezivinBoutabaa}.
This remark actually justifies the
the choice of considering two orders, instead of one as in the analytic theory.
\end{enumerate}
\end{rmk}

In the local case, both complex (\cf \cite{Beindex}, \cite{MarotteZhang}, \cite{ZhangFourier})
and $p$-adic (\cf \cite{BezivinBoutabaa}), the $q$-Gevrey order is not uniquely determined.
The global situation considered here is much more rigid: the same happens in  the differential case.

\begin{prop}
The orders of a given global $q$-Gevrey series $\sum_{n=0}^\infty a_nx^n\in K[[x]]\smallsetminus K[x]$
are uniquely determined.
\end{prop}

\begin{proof}
Suppose that $\sum_{n=0}^\infty a_nx^n$ is a global $q$-Gevrey series
of orders $(s_1,s_2)$ and $(t_1,t_s)$. By definition
$$
\sum_{n=0}^\infty {a_n\over\l(q^{n(n-1)\over 2}\r)^{s_1}\l([n]_q^!\r)^{s_2}}x^n
\hbox{\ and\ } \sum_{n=0}^\infty
{a_n\over\l(q^{n(n-1)\over 2}\r)^{t_1}\l([n]_q^!\r)^{t_2}}x^n
$$
have finite size. We have:
$$
\sum_{n=0}^\infty {a_n\over\l(q^{n(n-1)\over 2}\r)^{t_1}\l([n]_q^!\r)^{t_2}}x^n
=\sum_{n=0}^\infty \l(q^{n(n-1)\over 2}\r)^{s_1-t_1}\l([n]_q^!\r)^{s_2-t_2}
 {a_n\over\l(q^{n(n-1)\over 2}\r)^{s_1}\l([n]_q^!\r)^{s_2}}x^n\,.
$$
One observes that having finite size implies having finite radius of convergence
for all $v\in\cP$,
therefore for all $v$ such that $|q|_v\neq 1$
we must have:
$$
\limsup_{n\rightarrow\infty}
\l|\l(q^{n(n-1)\over 2}\r)^{s_1-t_1}\l([n]_q^!\r)^{s_2-t_2}\r|_v^{1/n}<\infty \,.
$$
If $|q|_v>1$ this implies:
$$
\limsup_{n\rightarrow\infty}
\l|q\r|_v^{{n-1\over 2}\l(s_1+s_2-(t_1+t_2)\r)}<\infty\,.
$$
Since for all $v\in\cP$ such that  $|q|_v<1$
the limit
$\limsup_{n\rightarrow\infty}|[n]_q^!|_v^{1/n}$
is bounded we get:
$$
\limsup_{n\rightarrow\infty}
\l|q\r|_v^{{n-1\over 2}\l(s_1-t_1\r)}<\infty\,.
$$
We deduce that necessarily $s_1+s_2\leq t_1+t_2$ and $t_1\leq s_1$, hence
$t_1\leq s_1$ and $s_2\leq t_2$.
Since the role of $(t_1,t_2)$ and $(s_1,s_2)$ is symmetric, one obviously obtain
the opposite inequalities  in the same way.
\end{proof}

\subsection{Changing $q$ in $q^{-1}$}
One can transform a $q$-difference equations in a $q^{-1}$-difference equations,
obtaining:

\begin{prop}\label{prop:cambiordine}
Let $f(x)\in K[[x]]$ be a global $q$-Gevrey series
of orders $\l(-s_1,-s_2\r)\in\Q^2$,
then $f(x)$ is a global $q^{-1}$-Gevrey series of orders
$(s_1+s_2,-s_2)$.
\par
In particular, if $f(x)$ is a global $q$-Gevrey series of orders
$(t_1,-t_2)$, with $t_1\geq t_2\geq 0$, then $f(x)$ is a global
$q^{-1}$-Gevrey series of negative orders $(-(t_1-t_2),-t_2)$.
\end{prop}

\begin{proof}
It is enough to write $f(x)$ in the form:
$$
f(x)=\sum_{n=0}^\infty
{a_n\over\l(q^{n(n-1)\over 2}\r)^{s_1}\l([n]_q^!\r)^{s_2}}x^n
=\sum_{n=0}^\infty
{a_n\over\l(q^{-{n(n-1)\over 2}}\r)^{-s_1-s_2}\l([n]_{q^{-1}}^!\r)^{s_2}}x^n\,,
$$
where $\sum_n a_nx^n$ is a convenient $G_q$-function.
\end{proof}

\subsection{Rescaling of the orders}\label{subsec:rescaling}
Clearly we can always look at a global $q$-Gevrey series of orders $(s,0)$ as
a global $q^t$-Gevrey series of orders $(s/t,0)$, for any $t\in\Q$, $t\neq 0$, the
holonomy condition being always satisfied:

\begin{lemma}\label{lemma:cambiordine}
Let $t\in\Q$, $t\neq 0$.
If $f(x)=\sum_{n=0}^\infty a_nx^n$ is solution of a $q$-difference equation
then it is solution of a $q^t$-difference equation.
\end{lemma}

\begin{proof}
If $f(x)$ is solution of a $q$-difference equation,
then it is also solution of a $q^{-1}$-difference equation.
Therefore we can suppose $t>0$.
Let $t={p\over r}$, with $p,r\in\Z_{>0}$. Since $f(x)$ is
solution of a $q$-difference operator,
we have:
$$
{\rm dim}_{K(x)}\sum_{i\geq 0}K(x)\sgq^i\l(f(x)\r)<+\infty\,.
$$
Then:
$$
{\rm dim}_{K(x)}\sum_{i\geq 0}K(x)\sg_{q^p}^i\l(f(x)\r)
={\rm dim}_{K(x)}\sum_{i\geq 0}K(x)\sgq^{ip}\l(f(x)\r)
\leq {\rm dim}_{K(x)}\sum_{i\geq 0}K(x)\sgq^i\l(f(x)\r)<+\infty\,,
$$
so $f(x)$ is solution of a $q^p$-difference operator.
Finally we can conclude since
$\sum_{i=0}^\nu a_i(x)f(q^{pi}x)=0$
implies that
$\sum_{i=0}^\nu a_i(x)f(\wtilde q^{t ir}x)=0$.
\end{proof}

Unfortunately, the same is not true for global $q$-Gevrey series
of orders $(0,s)$.
To prove it, one can calculate size of the series
$$
\Phi(x)=\sum_{n\geq 0}\frac{(\wtilde q;\wtilde q)_n^t}{(q;q)_n}x^n\,,
$$
where $\wtilde q$ is a $r$-th root of $q$, for some positive integer $r$,
$K=\Q(\wtilde q)$ and $t$ is an integer.
The Pochhammer symbols $(\wtilde q;\wtilde q)_n^t$ and $(q;q)_n$
are both polynomials in $\wtilde q^{1/2}$ of degree $tn(n+1)$
and $rn(n+1)$, respectively. If we want $\Phi(x)$ to have finite
size, we are forced to take $t\leq r$, so that it has
positive radius of convergence at any place $v$ such that $|q|_v>1$.
Notice that $\Phi(x)$ is convergent for any place $v$
such that $|q|_v<1$ and that the noncyclotomic places give a zero contribution
to the size.
As far as the cyclotomic places of $K$ is concerned, we obtain
$$
\sg_\cC(y)
=\limsup_{n\to\infty}\frac 1n\sum_{k=1}^{rn}
    \l(\l[\frac nk (k,r)\r]-t\l[\frac n k\r]\r)\lgp d^{-1}
\sim\limsup_{n\to\infty}\sum_{k=1}^{rn}\frac1k((k,r)-t)\log d^{-1}
\,.
$$
The limit above is infinite.

\section{Formal Fourier transformations}
\def\dep{d_p}

The following natural two $q$-analogues of the usual
formal Borel transformation
$$\begin{matrix}
(\cdot)^+:& K[[x]]&\longrightarrow &K[[z^{-1}]]\\\\
&\ds F=\sum_{n=0}^\infty a_nx^n &\longmapsto &
     \ds F^+=\sum_{n=0}^\infty [n]_q^!a_nz^{-n-1}
\end{matrix}
$$
and
$$\begin{matrix}
(\cdot)^\#:& K[[x]]&\longrightarrow &K[[z^{-1}]]\\\\
&\ds F=\sum_{n=0}^\infty a_nx^n &\longmapsto &
     \ds F^\#=\sum_{n=0}^\infty q^{n(n-1)\over 2}a_nz^{-n-1}
\end{matrix}\,,
$$
are equally considered in the literature on $q$-difference equations.
From an archimedean analytical point of view, they
are equivalent as soon as one works under the hypothesis that $|q|\neq 1$
(\cf \cite[\S8]{MarotteZhang} and \cite[Part II]{DVChanggui}).
As already noticed in \cite{andreannalsII}, from a global point of view,
$(\cdot)^+$ and $(\cdot)^\#$ have a completely different behavior:
for the same reason the definition of global $q$-Gevrey series involves two orders.
\par
Let $p=q^{-1}$ and let $\sgp:z\mapsto pz$,
$\dep=\frac{\sgp-1}{(p-1)z}$\footnote{This notation is a little bit ambiguous and we should rather
write $\sg_{p,z}$, $d_{p,z}$, $d_{q,x}$, etc. etc. Anyway the contest will be always clear enough not to be
obliged to specify the variable in the notation. }.
The Borel transformations that we have introduced above
have the following properties:

\begin{lemma}\label{lemma:relazionilaplace}
For all $F=\sum_{n=0}^\infty a_nx^n\in K[[x]]$ we have:
$$
\begin{array}{ll}
\ds(xF)^+=-p \dep F^+\,,&\ds(\dq F)^+=zF^+-F(0)\,,\\
\ds(xF)^\#=\frac pz\sgp F^\#\,,&\ds(\sgq F)^\#=p\sgp F^\#\,.
\end{array}
$$
\end{lemma}

\begin{proof}
We deduce the first equality using the relation:
$$
-p\dep{1\over z^n}=[n]_q {1\over z^{n+1}}\,.
$$
All the other formulas easily follow from the definitions.
\end{proof}

\par
For any $q$-difference operator
$\sum_{i=0}^N a_i(x)\sgq^i \in K(x)[\sgq]$
(resp. $\sum_{i=0}^N b_i(x)\dq^i$$\in K(x)[\dq]$) we set:
$$
{\rm deg}_{\sgq}\sum_{i=0}^\nu a_i(x)\sgq^i
=\sup\{i\in\Z :0<i<\nu,\ a_i(x)\neq 0\}\
$$
$$
\l(\hbox{resp. }
{\rm deg}_{\dq}\sum_{i=0}^\nu b_i(x)\dq^i
=\sup\{i\in\Z :0<i<\nu,\ b_i(x)\neq 0\}\r)
$$
Obviously we have $K(x)[\dq]=K(x)[\sgq]$ and
$\deg_{\dq}=\deg_{\sgq}$ (for explicit formulas \cf \cite[1.1.10]{DVInv} and \eqref{eq:ab} below).
The previous lemma justifies the definition of the formal Fourier
transformations below, acting on the skew rings
$K\l[x,\dq\r]$ and $K\l[x,\sgq\r]$:

\begin{defn}We call the maps:
$$\begin{matrix}
\cF_{q^+}&: K\l[x,\dq\r]&\longrightarrow & K\l[z,\dep\r] &\hskip 10pt\hbox{ and }\hskip 10pt&
   \cF_{q^\#}&: K\l[x,\sgq\r]&\longrightarrow & K\l[{1\over z},\sgp\r]\\\\
&\dq &\longmapsto & z &&&
   \sgq &\longmapsto & p\sgp\\\\
&x &\longmapsto & -p\dep&&&
   x &\longmapsto &{1\over qz}\sgp
\end{matrix}
$$
the {\it $q^+$-Fourier transformation} and
the {\it $q^\#$-Fourier transformation} respectively.
\end{defn}

\begin{rmk}
Let $\cF_p:K[z,\dep]\rightarrow K[x,\dq]$ and let
$\la: K\l[x,\dq\r]\rightarrow K\l[x, \dq\r]$, $\dq\mapsto-{1\over q}\dq$,
$x\mapsto-qx$
Then $\cF_{q^+}^{-1}=\la\circ\cF_{p^+}$.
\par
As far as $\cF_{q^\#}$ is concerned,
if ${\mathcal L}=\sum_{i=0}^\nu a_i({1\over z})\sg_p^i
\in K\l[{1\over z},\sgp\r]$ is such that
$\deg_{1\over z}a_i\l({1\over z}\r)\leq i$, there exists a unique
${\mathcal N}\in K\l[x,\sgq\r]$ such that
$\cF_{q^\#}({\mathcal N})={\mathcal L}$ and we note
$\cF_{q^\#}^{-1}({\mathcal L})={\mathcal N}$.
\end{rmk}

In the following lemma we verify that the formal Fourier
transformations we have just
defined are compatible with the Borel transformations $(\cdot)^+$ and
$(\cdot)^\#$:

\begin{lemma}
Let $F\in K[[x]]$ be a series solution of a
$q$-difference linear operator ${\mathcal N}\in K\l[x,\dq\r]$,
such that $\nu=\deg_{\dq}{\mathcal N}$
(resp. ${\mathcal N}\in K\l[x,\sgq\r]$). Then
$d_{q^{-1}}^\nu\circ\cF_{q^+}({\mathcal N})F^+=0$
(resp. $\cF_{q^\#}({\mathcal N})F^\#=0$).
\par
Inversely:
\begin{enumerate}
\item
If $F^+$ is a solution of ${\mathcal L}_1\in K\l[z,\dep\r]$, then
$\cF_{q^+}^{-1}({\mathcal L}_1)F=0$.
\item
If ${\mathcal L}_2\in K\l[{1\over z},\sgp\r]$ is such that
${\mathcal L}_2F^\#=0$, for all $n\in\N$, $n>>0$, we have:
$\cF_{q^\#}^{-1}(\sgp^n\circ{\mathcal L}_2)F=0$.
\end{enumerate}
\end{lemma}

\begin{proof}
We prove the statements for $(\cdot)^+$. The proof for
$(\cdot)^\#$ is quite similar.
We write $\cN$ in the form:
$$
{\mathcal N}=\sum_{j=0}^\nu \sum_{i=0}^N a_{i,j}x^i\dq^j\in  K\l[x,\dq\r]\,.
$$
Lemma \ref{lemma:relazionilaplace} implies that $\cF_{q^+}({\mathcal N})F^+$ is a polynomial
of degree less or equal to $\nu$, therefore $d_{q^{-1}}^\nu\circ\cF_{q^+}({\mathcal N})F^+=0$.
Let us now write $\cL_1$ as:
$$
{\mathcal L}_1=\sum_{j=0}^\nu \sum_{i=0}^N a_{i,j}z^i \dep^j\in
K\l[z,\dep\r]\,.
$$
Then $\l(\cF_{q^+}^{-1}({\mathcal L}_1)F\r)^+$ is a polynomial of
degree less or equal to $\nu$.
Hence we obtain:
$$
\dep^\nu \l(\cF_{q^+}^{-1}({\mathcal L}_1)F\r)^+=
\l((-qx)^\nu\cF_{q^+}^{-1}({\mathcal L}_1)F\r)^+=0
$$
and finally $(-qx)^\nu\cF_{q^+}^{-1}({\mathcal L}_1)F=0$.
\end{proof}

\begin{rmk}
In the following we will use the formal Fourier transformations above
composed with the symmetry $S:z\mapsto 1/x$:
\beq\label{eq:Fourier+sym}
\begin{matrix}
S\circ\cF_{q^+}&: K\l[x,\dq\r]&\longrightarrow & K\l[\frac 1x,x,\dq\r]
    &\hskip 10pt\hbox{ and }\hskip 10pt&
   S\circ\cF_{q^\#}&: K\l[x,\sgq\r]&\longrightarrow & K\l[x,\sgq\r]\\\\
&\dq &\longmapsto & \frac 1x&&&
   \sgq &\longmapsto & \frac 1q\sgq\\\\
&x &\longmapsto & x^2\dq&&&
   x &\longmapsto &{x\over q}\sgq
\end{matrix}\,.
\eeq
Notice that $S\circ\cF_{q^+}(\dq\circ x)=x\dq$.
\end{rmk}

\section{Action of the formal Fourier transformations on the Newton Polygon}

Let as consider a linear $q$-difference operator:
\beq\label{eq:operatore}
{\mathcal N}=\sum_{i=0}^\nu a_i(x)x^i\dq^i
=\sum_{i=0}^\nu b_i(x)\sgq^i\,,
\eeq
such that $b_j(x),a_j(x)\in K[x]$.
Applying formulas \cite[1.1.10]{DVInv}, we obtain:
\beq \label{eq:ab}
\begin{array}{rcl}
{\mathcal N}
    &=&\ds\sum_{j=0}^\nu b_j(x)\sum_{i=0}^j{j\choose i}_q
    (1-q)^i q^{i(i-1)/2}x^i\dq^i\\
&=&\ds\sum_{i=0}^\nu(1-q)^i q^{i(i-1)/2}
    \l(\sum_{j=i}^\nu {j\choose i}_qb_j(x)\r)x^i\dq^i\,.
\end{array}
\eeq
Therefore
$a_i(x)=(1-q)^i q^{i(i-1)/2}\sum_{j=i}^\nu {j\choose i}_q b_j(x)$.
\par
We recall the definition of the Newton-Ramis Polygon:

\begin{defn}
Let ${\mathcal N}=\sum_{i=0}^\nu a_i(x)x^i\dq^i
=\sum_{i=0}^\nu b_i(x)\sgq^i$ be
such that $b_j(x),a_j(x)\in K[x]$. Then we define
{\it the Newton-Ramis Polygon of
${\mathcal N}$ with respect to $\sgq$} (and we write $NRP_{\sgq}({\mathcal N})$)
(resp. {\it with respect to $\dq$} (and we write $NRP_{\dq}({\mathcal N})$))
to be the convex hull of the following set:
$$
\Cup_{b_i(x)\neq 0}\{(u,v)\in\R^2:
u=i,{\deg}_x b_i(x)\geq v\geq\ord_x b_i(x)\}\subset\R^2\,.
$$
$$
\l(\hbox{resp. }\Cup_{a_i(x)\neq 0}\{(u,v)\in\R^2:
u\leq i,{\deg}_x a_i(x)\geq v\geq\ord_x
a_i(x)\}\subset\R^2\r)\,.
$$
For an operator with rational coefficient $\cN$,
we set $NRP_{\sgq}(\cN)=NRP_{\sgq}(f(x)\cN)$ and $NRP_{\dq}(\cN)=NRP_{\dq}(f(x)\cN)$, where $f(x)$
is a polynomial in $K[x]$ such that $f(x)\cN$ can be written as above.
In this way the Newton-Ramis polygon is defined up to a vertical shift,
so that its slopes are actually well-defined.
\end{defn}

\begin{lemma}
We have:
$$
NRP_{\dq}({\mathcal N})=\Cup_{(u_0,v_0)\in NRP_{\sgq}({\mathcal N})}
\{(u,v_0)\in\R^2: u\leq u_0\}\,.
$$
\end{lemma}

\begin{proof}
The statement follows from \eqref{eq:ab}.
\end{proof}

The following proposition describes the behavior of the Newton-Ramis Polygon
with respect to $\cF_{q^+}$ and $\cF_{q^\#}$.

\begin{prop}\label{prop:newtonfourier}
The map\footnote{To make the notation clear,
we underline that we denote $NRP_{\sgp}\l(\cF_{q^\#}\l({\mathcal N}\r)\r)$
the Newton-Ramis Polygon of $\cF_{q^\#}\l({\mathcal N}\r)$ defined with respect to
$z$ and $\sgp$ and $NRP_{\dep}\l(\cF_{q^+}\l({\mathcal N}\r)\r)$
the Newton-Ramis Polygon of $\cF_{q^+}\l({\mathcal N}\r)$  defined with respect to
$z$ and $\dep$.}:
$$
\begin{matrix}
NRP_{\sgq}({\mathcal N})&\longrightarrow
   &NRP_{\sgp}\l(\cF_{q^\#}\l({\mathcal N}\r)\r)\\\\
(u,v)&\longmapsto &(u+v,-v)
\end{matrix}
\hskip 20 pt
\l(\hbox{resp.\,\,}
\begin{matrix}
NRP_{\dq}({\mathcal N})&\longrightarrow
   &NRP_{\dep}\l(\cF_{q^+}\l({\mathcal N}\r)\r)\\\\
(u,v)&\longmapsto &(u+v,-v)\end{matrix}\r)
$$
is a bijection between $NRP_{\sgq}({\mathcal N})$
and $NRP_{\sgp}\l(\cF_{q^\#}\l({\mathcal N}\r)\r)$
(resp. $NRP_{\dq}({\mathcal N})$ and $NRP_{\dep}$ $\l(\cF_{q^+}\l({\mathcal N}\r)\r)$).
\end{prop}

\begin{proof}
As far as $\cF_{q^\#}$ is concerned, it is enough to notice that:
$$
\cF_{q^\#}\l(\sum_{i=0}^\nu\sum_{j=0}^N b_{i,j}x^j\sgq^i\r)
=\sum_{i=0}^\nu\sum_{j=0}^N {b_{i,j}\over q^{j(j-3)/2}q^i}{1\over z^j}
\sgp^{i+j}\,.
$$
Let
$$
{\mathcal N}=\sum_{i=0}^\nu\sum_{j=0}^N a_{i,j}x^j\dq^i\,.
$$
We have:
$$\begin{array}{rcl}
\cF_{q^+}\l({\mathcal N}\r)
&=&\ds\sum_{i=0}^\nu\sum_{j=0}^N {(-1)^ja_{i,j}\over q^j}
  \dep^j\circ z^i\\
&=&\ds\sum_{i=0}^\nu\sum_{j=0}^\nu\sum_{h=0}^j{(-1)^ja_{i,j}\over q^j}
{j\choose h}_q{[i]_q^!\over[h-i]_q^!}q^{(j-h)(i-h)}z^{i-h}\dep^{j-h}\,.
\end{array}
$$
Then if $(i,j-i)\in NRP_{\dq}({\mathcal N})$ we have:
$$
(j-h,i-j)\in NRP_{\dep}\l(\cF_{q^+}\l({\mathcal N}\r)\r)
\hbox{ for all $h=0,\dots,j$.}
$$
The statement follows from this remark.
\end{proof}

By convention,
the vertical sides of $NRP_{\sgq}({\mathcal N})$)
(resp. $NRP_{\dq}({\mathcal N})$) have slope $\infty$.
The opposite of the finite slopes of the ``upper part'' of $NRP_{\sgq}({\mathcal N})$
are the slopes at $\infty$ of ${\mathcal N}$ while the finite slopes of
the ``lower part'' are the slopes of ${\mathcal N}$ at $0$.

\begin{cor}\label{cor:newtonfourier}
In the notation of the previous proposition,  $\cF_{q^\#}$ (resp. $\cF_{q^+}$)
acts in the following way on the slopes of the Newton-Ramis Polygon:
$$\begin{matrix}
&\l\{\hbox{slopes of $NRP_{\sgq}({\mathcal N})$}\r\}&\longrightarrow
        &\l\{\hbox{slopes of $NRP_{\sgp}\l(\cF_{q^\#}\l({\mathcal N}\r)\r)$}\r\}\\ \\
\Big(\hbox{resp. }&\l\{\hbox{slopes of $NRP_{\dq}({\mathcal N})$}\r\}&\longrightarrow
 &\l\{\hbox{slopes of $NRP_{\dep}\l(\cF_{q^\#}\l({\mathcal N}\r)\r)$}\r\}&\Big)\\ \\
    &\la&\longmapsto &-{\la\over 1+\la}\\
    &\infty&\longmapsto &-1\end{matrix}\,.
$$
\end{cor}

\section{Solutions at points of $K^\ast$}

We have described what happens
at zero and at $\infty$ when the Fourier transformations act. Now we want
to describe what happens at a point $\xi\in K^\ast={\mathbb P}^{1}(K)\smallsetminus\{0,\infty\}$.

\smallskip
To construct some formal solutions of our $q$-difference operators at
$\xi\in K^\ast$, we are going to consider a
ring defined as follows (\cf \cite[\S1.3]{DVdwork}).
For any $\xi\in K$ and any nonnegative integer $n$, we consider the polynomials
$$
T^q_n(x,\xi)=x^n\l(\frac \xi x;q\r)_n=
(x-\xi)(x-q\xi)\cdots(x-q^{n-1}\xi)\,.
$$
One verifies directly that for any $n\geq 1$
$$
\dq T^q_n(x,\xi)=[n]_qT^q_{n-1}(x,\xi)\,
$$
and $\dq T^q_0(x,\xi)=0$. The product
$T^q_n(x,\xi)T^q_m(x,\xi)$ can be written as a linear combination
with coefficients in $K$ of $T^q_0(x,\xi),T^q_1(x,\xi),\dots,T^q_{n+m}(x,\xi)$
(\cf \cite[\S 1.3]{DVdwork}).
It  follows that we can define the ring:
$$
K[[x-\xi]]_q=
\l\{\sum_{n\geq 0}a_n T^q_n(x,\xi): a_n\in K\r\}\,,
$$
with the obvious sum and the Cauchy product described above, extended by linearity.
The ring $K[[x-\xi]]_q$ is a $q$-difference algebra with the natural action of $\dq$.
Notice that in general it makes no sense to look at the sum of those series. Nevertheless,
they can be evaluated at the point of the set $\xi q^{\Z_{\geq 0}}$, and they are actually
in bijective correspondence with the sequences $\{f(\xi q^n)\}_{n\in\Z_{\geq 0}}\in\C^\N$.

\begin{prop}\label{prop:fouriersinguno}
Let ${\mathcal N}\in K\l[x,\dq\r]$ be a linear $q$-difference operator such that
$NRP_{\dq}({\mathcal N})$ has only the zero slope at $\infty$\footnote{or equivalently,
$NRP_{\dq}({\mathcal N})$ has no negative slopes.};
then the operator $\cF_{q^+}{\mathcal N}$ has a basis of solution in $K[[z-\xi]]_p$ for all
$\xi\in K^\ast$.
\end{prop}

\begin{proof}
The hypothesis on the Newton Polygon of ${\mathcal N}$ at $\infty$
implies that we can write $\cN$ in the following form
$$
{\mathcal N}=\sum_{i=0}^\nu\sum_{j=0}^N a_{i,j}x^j\dq^i\,,
$$
with $a_{i,N}=0$ for all $i=0,\dots,\nu-1$ and $a_{\nu,N}\neq 0$.
This implies that the coefficient of $\dep^N$ in
$$\begin{array}{rcl}
\cF_{q^+}({\mathcal N})
&=&\ds\sum_{i=0}^\nu\sum_{j=0}^N  a_{i,j}
  \l(-p\dep\r)^j\circ z^i\\
&=&\ds\sum_{j=0}^{N-1} \sum_{i=0}^\nu c_{j,i}z^i \dep^j
  + a_{\nu,N}\l(-q\r)^{\nu-N}z^\nu\dep^N
\end{array}
$$
does not have any zero in the set $\{q^n\xi:n\in\Z_{>0}\}$.
Using the fact that $\dep T^p_n(z,\xi)=[n]_pT^p_{n-1}(z,\xi)$ and
that $zT^p_n(z,\xi)=T^p_{n+1}(z,\xi)+p^n\xi T^p_n(z,\xi)$, a basis of solutions of
$\cF_{q^+}({\mathcal L})$ in $K[[z-\xi]]_p$ can be constructed working with the recursive
relation induced by $\cF_{q^+}({\mathcal L})y=0$ on the coefficients of a generic solution of the form
$\sum_n \a_nT^p_n(z,\xi)$.
\end{proof}

\begin{cor}
For any $\cN\in K[z,\dep]$ (resp. $\cN\in K[x,\dq]$, $\cN\in K[z,\dep]$) having only the zero slope at
$\infty$,
the operator $\cF_{q^+}^{-1}({\mathcal N})$
(resp. $S\circ\cF_{q^+}({\mathcal N})$, $S\circ \cF_{q^+}^{-1}({\mathcal N})$)
has a basis of solution in $K\[[x-\xi\]]_q$ for any $\xi\in K^\ast$.
\end{cor}

\begin{proof}
The statement follows from the remark that $\cF_{q^+}^{-1}({\mathcal N})=\la\circ\cF_{p^+}({\mathcal N})$
and that the symmetry $S:z\mapsto 1/x$ does not changes the kind of singularity at the points of
$K^\ast$.
\end{proof}

An analogous property holds for $\cF_{q^\#}^{-1}$:

\begin{prop}\label{prop:fouriersingdue}
Let ${\mathcal L}=\sum_{i=0}^\nu a_i\l({1\over z}\r)\sgp^i
\in K\l[{1\over z},\sgp\r]$ such that
$\deg_{1\over z}a_i({1\over z})\leq i$.
We suppose that
$$
N=\ord_{1\over z}a_\nu\l({1\over z}\r)\leq \ord_{1\over z}a_i\l({1\over z}\r)\,,
$$
for all $i=0,\dots,\nu-1$\footnote{or equivalently, that
$NRP_{\dq}(S\circ\cL)$ does not have any positive slope.}.
Then $\cF_{q^\#}^{-1}({\mathcal L})$ has a basis of solution in $K[[x-\xi]]_q$ for all
$\xi\in K^\ast$.
\end{prop}

\begin{proof}
We call $a_{\nu,N}\in K$ the coefficients of ${1\over z}^N$ in
$a_\nu\l({1\over z}\r)$. Then we have:
$$
\cF_{q^\#}^{-1}({\mathcal L})=\sum_{i=0}^{\nu-N-1}b_i(x)\sgq^i+
a_{\nu,N}x^N\sgq^{\nu-N}\,.
$$
One ends the proof as above.
\end{proof}


\section{Structure theorems}
\label{sec:gev}

Inspired by \cite{andreannalsI},
we want to characterize $q$-difference operators
killing a global $q$-Gevrey series of orders $(-s_1,-s_2)$,  with
$(s_1,s_2)\in\cZ:=\Q_{\geq 0}\times\Z_{\geq 0}\smallsetminus\{(0,0)\}$.

\smallskip
The skew polynomial ring $K(x)[\dq]$ is euclidean with respect to
$\deg_{\dq}$.
It follows that, if we have a formal power series $y$ solution of
a $q$-difference operator, we can find a $q$-difference operator $\cL$
killing $y$ and such that $\deg_{\dq}\cL$ is minimal.
All the other linear $q$-difference
operators killing $y$, minimal with respect to $\deg_{\dq}$, are of the form
$f(x)\cL$, with $f(x)\in K(x)$.
By abuse of language, we will call the minimal degree operator $\cL\in K[x,\dq]$
(resp. $K[x,\sgq]$) with no common factors in the coefficients \emph{the} minimal
operator killing $y$.

\begin{rmk}
Let $y(x)\in K[[x]]$ be a formal power series solution of the
linear $q$-difference operator
$\cL_q=\sum_{i=0}^\nu a_i(x)\sgq^i$.
We choose $\cL_q$ such that
$\deg_{\sgq}\cL_q$ is minimal.
Then for all positive integers $r$ the operator
$\cL_{q^{1/r}}=\sum_{i=0}^\nu a_i(x)\sg_{q^{1/r}}^{ir}$ is the
minimal $q^{1/r}$-difference operator killing $y(x)$.
Moreover if $\la$ is a slope of $NRP_{\sgq}\l(\cL_q\r)$
(resp. $NRP_{\dq}\l(\cL_q\r)$)
then $\la/r$ is a slope of $NRP_{\sgq}\l(\cL_{q^{1/r}}\r)$
(resp. $NRP_{\dq}\l(\cL_{q^{1/r}}\r)$).
In fact, let $\cL$ be a $q^{1/r}$-difference operator killing $y(x)$, minimal with respect to
$\deg_{\sg_{q^{1/r}}}$. Then $\cL_{q^{1/r}}$ is a factor of $\cL$ in $K(x)[\sg_{q^{1/r}}]$,
hence $\deg_{\sg_{q^{1/r}}}\cL\leq r\deg_{\sgq}\cL_q$.
On the other side we have:
$$
\dim_{K(x)}\sum_{i\geq 0}K(x)\sg_{q^{1/r}}^i(y)\geq
\dim_{K(x)}\sum_{i\geq 0}K(x)\sg_{q^{1/r}}^{ir}(y)=
\dim_{K(x)}\sum_{i\geq 0}K(x)\sgq^i(y)\,,
$$
therefore $\deg_{\sg_{q^{1/r}}}\cL\geq r\deg_{\sgq}\cL_q$.
\end{rmk}

We recall the statement of Corollary \ref{cor:regularity}, which is the starting point for this second
part of the paper:

\begin{prop}\label{prop:regop}
Let $F\in K[[x]]$ be a global $q$-Gevrey series of orders $(0,0)$ and
$\cL\in K[x,\dq]$ the minimal $q$-difference operator
such that $\cL F=0$. Then $\cL$ is regular singular.
\end{prop}

Using the formal $q$-Fourier transformations introduced in the previous section,
we will deduce the structure
theorems below from Proposition \ref{prop:regop}.

%
%
%

\begin{thm}\label{thm:sing}
Let $F\in K[[x]]\smallsetminus K[x]$ be
a global $q$-Gevrey series of orders $(-s_1,-s_2)$,
with $(s_1,s_2)\in\cZ$
and $\cL\in K\l[x,\dq\r]$ be the minimal linear $q$-difference operator
such that $\cL F=0$.
Then $\cL$ has the following properties:
\par\noindent
- the set of finite slopes of the Newton Polygon $NRP_{\dq}(\cL)$ is
$\l\{-1/(s_1+s_2),0\r\}$;
\par\noindent
- for all $\xi\in K^\ast$,
the $q$-difference operator $\cL$ has a basis
of solutions in $K[[x-\xi]]_q$.
\end{thm}

\begin{proof}
Let us write the formal power series $F$ in the form:
$$
F=\sum_{n=0}^\infty {a_n\over\l(q^{n(n-1)\over 2}\r)^{s_1}\l([n]_q^!\r)^{s_2}}x^n\,,
$$
where $\sum_{n=0}^\infty a_nx^n$ is a $G_q$-function.
Let $\widetilde F(x)=\sum_{n=0}^\infty a_nx^{n+s_2}$;
then the series $\widetilde F$ has finite size, therefore there exists a regular singular
$q$-difference operator $\cL\in K[x,\sgq]$ such that $\cL\widetilde F=0$.
The polygon $NRP_{\sgq}(\cL)$ has only the zero slope (apart from the infinite slopes).
\par
Let $\mathcal S$ be the symmetry with respect to the origin:
$$\begin{matrix}
\mathcal S:& x&\longmapsto& 1/z\\
 &\sgq & \longmapsto & \sgp
\end{matrix}
\,.
$$
Remark that the operator $\cF_{q^+}^{-1}\circ \mathcal S\l(\cL\r)$
kill the formal power series
$\sum_{n=0}^\infty {a_n\over [n]_q^!}x^{n+s_2-1}$.
The polygon $NRP_{\dep}(\mathcal S(\cL))$ is obtained
by $NRP_{\dq}(\cL)$ applying
a symmetry with respect to the line $v=0$.
It follows from Proposition \ref{cor:newtonfourier} that
the set of finite slopes of $NRP_{\dq}\big(\cF_{q^+}^{-1}\circ \mathcal S\l(\cL\r)\big)$
is $\{0,-1\}$.
Iterating $s_2$ times this reasoning, we obtain a $q$-difference operator
$\widetilde\cL=
\cF_{q^+}^{-1}\circ \mathcal S\circ\dots\circ\cF_{q^+}^{-1}\circ \mathcal S\l(\cL\r)$,
such that the set of finite slopes of $NRP_{\dq}\l(\widetilde\cL\r)$ is
$\{0,-1/s_2\}$.
We obtain:
$$
\widetilde\cL\l(\sum_{n=0}^\infty {a_n\over\l([n]_q^!\r)^{s_2}}x^{n}\r)=0\,.
$$
Because of \S\ref{subsec:rescaling},
we can now suppose that $s_1$ is actually a positive
integer.
We conclude the proof applying the same argument to
$\ol\cL=
\l(\cF_{q^\#}^{-1}\circ\circ \mathcal S\r)\circ\dots\circ
\l(\cF_{q^\#}^{-1}\circ \mathcal S\r)\l(\sgq^n\circ\widetilde\cL\circ x^{s_1}\r)$,
for a suitable $n\in\Z_{\geq 0}$, and to the Newton-Ramis Polygon
defined with respect to $\sgq$.
We know that $\ol\cL F=0$.
\par
The operator $\cL$ is a factor of $\ol\cL$ in $K(x)[\sgq]$.
We know (\cf for instance \cite{sauloyfiltration})
that the slopes of the Newton Polygon of $\cL$ at zero
(resp. $\infty$) are slopes of the Newton Polygon of $\ol\cL$ at zero
(resp. $\infty$). To obtain the desired result on the slopes of
$NRP_{\dq}(\cL)$ one has to notice that $\ol\cL$ must have a positive slope at $\infty$
because of \cite[Theorem 4.8]{RamisToulouse}.
As far as $\xi\in K^\ast$
is concerned, the operator $\ol\cL$ has a basis
of solutions at $\xi$ in $K[[x-\xi]]_q$
(\cf Propositions \ref{prop:fouriersinguno} and \ref{prop:fouriersingdue}),
therefore the same is true for $\cL$.
\end{proof}

Proposition \ref{prop:newtonfourier} implies that for a global $q$-Gevrey series of orders $(-s_1,0)$
we have actually proved a more precise result:

\begin{thm}\label{thm:singbis}
Under the hypothesis of the previous theorem, we assume that $s_2=0$.
Then $\cL$ has the following properties:
\par\noindent
- the set of finite slopes of $NP_{\sgq}(\cL)$ is $\l\{0,-1/s_1\r\}$;
\par\noindent
- for all $\xi\in  K^\ast$,
the $q$-difference operator $\cL$ has a basis
of solutions in $K[[x-\xi]]_q$.
\end{thm}

Changing $q$ in $q^{-1}$ we get the corollary:

\begin{cor}
Let $F\in K[[x]]\smallsetminus K[x]$ be
a global $q$-Gevrey series of orders $(s_1,-s_2)$,
with $(s_1,s_2)\in\Q\times\Z$, such that $s_1\geq s_2\geq  0$
and either $s_1\neq s_2$ or $s_2\neq 0$.
Let $\cL\in K\l[x,\sgq\r]$ be the minimal linear $q$-difference operator
such that $\cL F=0$.
Then $\cL$ has the following properties:
\par\noindent
- the set of finite slope of $NP_{\dep}(\cL)$ is $\{0,1/s_1\}$
\par\noindent
- for all $\xi\in  K^\ast$,
the $q$-difference operator $\cL$ has a basis
of solutions in $K[[x-\xi]]_p$.
\end{cor}

\begin{proof}
It follows by Proposition \ref{prop:cambiordine}, taking into account that
when one changes $q$ in $q^{-1}$, the slopes of the Newton Polygon change sign.
\end{proof}

Following \cite{andreannalsII} we can characterize the apparent singularities of
such a $q$-difference equation:

\begin{thm}\label{thm:teogev}
Let $F\in K[[x]]\smallsetminus K[x]$
be a global $q$-Gevrey series of orders $(-s_1,-s_2)$,
with $(s_1,s_2)\in\cZ$.
We fix a point $\xi\in K^\ast$.
For all $v\in\cP$ such that $|q|_v>1$
we suppose that the $v$-adic function $F(x)$ has a zero at $\xi$.
Let $\cL\in K\l[x,\dq\r]$ be the minimal linear $q$-difference operator
such that $\cL F=0$.
Then $\cL$ has a basis of solution in
$$
(x-\xi)K[[x-q\xi]]_q=\l\{\sum_{n=1}^\infty a_n(x-\xi)_n:
a_n\in K\r\}\,.$$
\end{thm}

The proof is based on the following
lemma, which is an analogue of \cite[Lemme 2.1.2]{andreannalsII}
(\cf also \cite[Lemma 4.4.2]{andreannalsII}).

\begin{lemma}\label{lemma:zero}
Let $F$ be a global $q$-Gevrey series of orders $(-s_1,-s_2)$,
with $s_1,s_2\in\Q_{\geq 0}\times\Z_{\geq 0}$.
We fix a point $\xi\in K^\ast$.
For all $v\in\cP$ such that $|q|_v>1$
we suppose that the $v$-adic entire function $F(x)$ has a zero at $\xi$.
Then $G=(x-\xi)^{-1}F$ is a global $q$-Gevrey series of orders
$(-s_1,-s_2)$.
\end{lemma}

\begin{proof}[Proof of Theorem \ref{thm:teogev}]
We fix some notation:
$$
F=\sum_{n=0}^\infty{a_n\over q^{s_1{n(n-1)\over 2}}{[n]_q^!}^{s_2}}x^n\,,\
G=\sum_{n=0}^\infty{b_n\over q^{s_1{n(n-1)\over 2}}{[n]_q^!}^{s_2}}x^n\,,
$$
$$
\widetilde h(n,v,F)=\sup_{s\leq n}|a_s|_v
\hbox{\ and }
\widetilde h(n,v,G)=\sup_{s\leq n}|b_s|_v\,.
$$
Since $\frac{1}{x-\xi}=-\sum_{n\geq 0}\frac{x^n}{\xi^{n+1}}$, we obtain:
$$
b_n=-\sum_{k=0}^n\l(q^{{n(n-1)\over 2}-{k(k-1)\over 2}}\r)^{s_1}
\l({[n]_q^!\over [k]_q^!}\r)^{s_2}\xi^{k-n-1}a_k
$$
and therefore:
$$
\limsup_{n\rightarrow\infty}{1\over n}\sum_{|q|_v\leq 1}\widetilde h(n,v,G)
\leq \limsup_{n\rightarrow\infty}{1\over n}\sum_{|q|_v\leq 1}\widetilde h(n,v,F)
+\sum_{|q|_v\leq 1}|\xi|_v\,.
$$
To conclude it is enough to prove that $G$ is a local $q$-Gevrey series
of order $s_1+s_2$ for all
$v\in\cP$ such that $|q|_v>1$.
This follows from \cite[Prop. 2.1]{RamisToulouse}, since $F$ and $G$ have the same growth at
$\infty$, because $F$ has a zero at $\xi$.
\end{proof}

\begin{proof}
Let $G=(x-\xi)^{-1}F$ and $\cL$ be the minimal linear $q$-difference operator such that
$\cL F=0$;
then $\cL\circ(x-\xi)$ is the minimal linear
$q$-difference operator  such that $\cL\circ(x-\xi)(G)=0$. By
Lemma \ref{lemma:zero} and Theorem \ref{thm:sing},
$\cL\circ(x-\xi)$ has a basis of solution in $K[[x-q\xi]]_q$,
therefore the operator
$\cL$ has a basis of solution in $(x-\xi) K[[x-q\xi]]_q$.
\end{proof}

Once again, switching $q$ into $q^{-1}$ we obtain the corollary:

\begin{cor}\label{cor:teogevbis}
Let $F\in K[[x]]\smallsetminus K[x]$
be a global $q$-Gevrey series of orders $(s_1,-s_2)$,
with $s_1,s_2\in\Q\times\Z$, $s_1\geq s_2\geq 0$ and either $s_1\neq s_2$ or $s_2\neq 0$.
We fix a point $\xi\in K^\ast$.
For all $v\in\cP$ such that $|q|_v<1$
we suppose that the $v$-adic function $F(x)$ has a zero at $\xi$.
Let $\cL\in K\l[x,\dq\r]$ be the minimal linear $q$-difference operator
such that $\cL F=0$. Then $\cL$ has a basis of solution in
$$
(x-\xi)K[[x-p\xi]]_p\,.$$
\end{cor}

\begin{proof}
It follows from Proposition \ref{prop:cambiordine} and Theorem \ref{thm:teogev}.
\end{proof}

We conclude the section with an example:

\begin{exa}
Let us consider the $q$-exponential series $E_q(x)=\sum_{n\geq 0}{x^n\over [n]_q^!}$, solution of
the equation $\dq y=y$. A classical formula (\cf \cite[1.3.16]{GR}) says that for
$|q|_v>1$ the series $E_q(x)$ can be written as an infinite product:
$$
E_q(x)=\l(-x(1-q^{-1});q^{-1}\r)_\infty:=\prod_{k=0}^\infty\l(1-x{1-q\over q^{k+1}}\r)\,,
$$
hence $E_q({q\over 1-q})=0$ for all $v$ such that $|q|_v>1$.
Let us consider formal q-series:
$$
G(x)={E_q(x)\over x-{q\over 1-q}}={q-1\over q}E_q\l({x\over q}\r)\,.
$$
Obviously, $q\dq G(x)-G(x)=0$ and actually:
$$
\l(\dq-1\r)\circ\l(x-{q\over 1-q}\r)G(x)=\l(x-{1\over 1-q}\r)\l(q\dq-1\r)G(x)=0\,.
$$
Since $\sum_{n\geq 0}{q^{-n}\over [n]_q^!}
T^q_n\l(x,{q^2\over 1-q}\r)\in K[[x-{q^2\over 1-q}]]_q$
is a formal solution of $q\dq y=y$, the series
$$
\l(x-{q\over 1-q}\r)\sum_{n\geq 0}{q^{-n}\over [n]_q^!}T^q_n\l(x,{q^2\over 1-q}\r)
\in \l(x-{q\over 1-q}\r)K\l[\l[x-{q^2\over 1-q}\r]\r]_q
$$
is a formal solution of $\dq y=y$.
\end{exa}

\section{An irrationality result for global $q$-Gevrey series of negative orders}

In this section we are going to give a simple criteria to
determine the $q$-orbits where a global
$q$-Gevrey series does {\it not} satisfy the hypothesis of Theorem \ref{thm:teogev}.
We will deduce an irrationality result for values of a global $q$-Gevrey series
$F(x)\in K[[x]]\smallsetminus K[x]$ of negative orders.

\begin{rmk}
The arithmetic Gevrey series theory in the differential case
has applications to transcendence theory (\cf \cite{andreannalsII}).
In the global $q$-Gevrey series framework this can not be true, since the set
of global $q$-Gevrey series has only a structure of $\ol k(q)$-vector
space. We mean that the product of two global $q$-Gevrey series of nonzero orders doesn't need to be a global
$q$-Gevrey series, as the following example shows:
$$
e_q(x)^2=\l(\sum_{n=0}^\infty{x^n\over[n]_q^!}\r)^2=
\sum_{n=0}^\infty\l(\sum_{k=0}^n{n\choose k}_q\r){x^n\over[n]_q^!}\,.
$$
In fact, because of the estimate at the cyclotomic places $e_q(x)^2$ should be a
global $q$-Gevrey series of order $(0,-1)$, while the local $q$-Gevrey order at places $v\in\cP_\infty$
such that $|q|_v>1$  is $2$.
For this reason a global $q$-Gevrey series theory can only
have applications to the irrationality theory.
\end{rmk}

Let
$$
\cL=a_\nu(x)\sgq^\nu+\dots+a_1(x)\sgq+a_0(x)\in K\l[x,\sgq\r]\,,
$$
and let $u_0,\dots,u_{\nu-1}$ a basis of solution of $\cL$ is a convenient
$q$-difference algebra extending $K(x)$.
The \emph{Casorati matrix}
$$
\cU=\begin{pmatrix}
u_0&\cdots &u_{\nu-1}\\
\sgq u_0&\cdots &\sgq u_{\nu-1}\\
\vdots&\ddots&\vdots\\
\sgq^{\nu-1}u_0&\cdots &\sgq^{\nu-1}u_{\nu-1}
\end{pmatrix}\,,
$$
is a fundamental solution of the $q$-difference system
$$
\sgq\cU=
\begin{pmatrix}
\begin{matrix}0\\ \vdots\\ 0\end{matrix}&\vrule &{\mathbb I}_{\nu-1}\\
    \hrulefill&\hrulefill&\hskip -30pt\hrulefill\\
    -\frac{a_0(x)}{a_\nu(x)}&\vrule
    &\begin{matrix}-\frac{a_1(x)}{a_\nu(x)}&\dots&-\frac{a_{\nu-1}(x)}{a_\nu(x)}\\
\end{matrix}\end{pmatrix}\cU\,,
$$
so that $\cC=\det \cU$ is solution of the equation:
$$
\sgq \cC=(-1)^\nu{a_0(x)\over a_\nu(x)}\cC\,.
$$
Notice that the ``$q$-Wronskian lemma''
(\cf for instance \cite[\S 1.2]{DVInv})
implies that the determinant of the Casorati matrix of a basis of solutions of an operator $\cL$
is nonzero.

\begin{prop}
Let $F\in K[[x]]\smallsetminus K[x]$
be a global $q$-Gevrey series of orders $(-s_1,-s_2)$,
with $s_1,s_2\in\cZ$.
We fix a point $\xi\in K^\ast$.
Let $\cL=a_\nu(x)\sgq^\nu+\dots+a_1(x)\sgq+a_0(x)\in K\l[x,\sgq\r]$
be the minimal $q$-difference operator such that $\cL F=0$.
If $F(x)$ has a zero at $\xi$ for all $v$ such that $|q|_v>1$,
then there exists an integer $m\geq 0$ such that $q^m\xi$ is a zero of $a_0(x)$.
\end{prop}

\begin{proof}
The determinant of the Casorati matrix of a basis of solutions of $\cL$
satisfies the equation
$$
y(qx)=(-1)^\nu\frac{a_\nu(x)}{a_0(x)}y(x)\,.
$$
On the other hand we know that $\cL$ has a basis of solution
$u_0,\dots,u_{\nu-1}\in (x-\xi)K[[x-q\xi]]$.
This means that the $u_i$'s are formal series of the form
$\sum_{n\geq 1}a_nT_n^q(x,\xi)$, for some $a_n\in K$.
Since $(qx-\xi)=q(x-q^{n-1}\xi)+(q^n-1)\xi$, one obtain that
$$
\sgq\l(\sum_{n\geq 1}a_nT_n^q(x,\xi)\r)
=qa_1+\sum_{n\geq 1}\l(q^n a_n+q^{n+1}a_{n+1}\xi(q^n-1))\r)T_n^q(x,\xi)\,.
$$
This implies that the determinant $\cC$ of the Casorati matrix of
$u_0,\dots,u_{\nu-1}$ is an element of $(x-\xi)K[[x-q\xi]]_q$.
Let $m\geq 1$ be the larger integer such that $\cC\in T_m^q(x,\xi)K[[x-q^m\xi]]_q$.
The formula above implies that
$\sgq\cU\in T_{m-1}^q(x,\xi)K[[x-q^{m-1}\xi]]_q\smallsetminus
T_m^q(x,\xi)K[[x-q^m\xi]]_q$, and therefore
that $q^{m-1}\xi$ is a zero of $a_0(x)$.
\end{proof}

In the same way we can prove the following result:

\begin{cor}
Let $F\in K[[x]]\smallsetminus K[x]$
be a global $q$-Gevrey series of orders $(s_1,-s_2)$,
with $s_1,s_2\in\Q\times\Z$, $s_1\geq s_2\geq 0$ and either $s_1\neq s_2$ or $s_2\neq 0$.
We fix a point $\xi\in K^\ast$.
Let $\cL=a_\nu(x)\sgq^\nu+\dots+a_1(x)\sgq+a_0(x)\in K\l[x,\sgq\r]$
be the minimal linear $q$-difference operator  such that $\cL F=0$.
If $F(x)$ has a zero at $\xi$
for all $v\in\cP$ such that $|q|_v<1$ then there exists an integer $m\leq -\nu$ such that
$q^m\xi$ is a zero of $a_\nu(x)$.
\end{cor}

\begin{proof}
It follows from
Proposition \ref{prop:cambiordine} that $F(x)$ is a global $q^{-1}$-Gevrey series
of negative orders $(-(s_1-s_2),-s_2)$ and the
minimal linear $q^{-1}$-difference operator killing $F(x)$ is
$a_\nu(q^{-\nu}x)+\dots+a_1(q^{-\nu}x)\sg_{q^{-1}}^{\nu-1}+a_0(q^{-\nu}x)\sg_{q^{-1}}^\nu$.
\end{proof}

\begin{exa}
Let us consider the field $K=k(q)$ and the Tchakaloff series:
$$
T_q(x)=\sum_{n\geq 0}{x^n\over q^{n(n-1)/2}}\,.
$$
Together with $E_q(x)$, $T_q(x)$ is a $q$-analogue of the exponential function.
The minimal linear $q$-difference equation killing $T_q(x)$
is
$$
\cL=(\sgq-1)\circ(\sgq-qx)
=(\sgq-q^2x)\circ(\sgq-1)
=\sgq^2-(1+q^2x)\sgq+q^2x\,.
$$
Notice that $1,T_q(x)$ is a basis of solutions of $\cL$ at zero.
We conclude that
$T_q(\xi)\neq 0$ for all $\xi\in K^\ast$, as the value a $q^{-1}$-adic entire analytic function,
{\it i.e.} the hypothesis of Theorem \ref{thm:teogev} are never satisfied.
\par
In particular, let $K= k(\wtilde q)$, where $\wtilde q^r=q$ for some positive integer $r$.
For any $\xi\in k(\wtilde q)$, $\xi\neq 0$, the $\wtilde q^{-1}$-adic
value $T_q(\xi)$ of $T_q(x)$ at $\xi$ can be formally written as a Laurent series in
$k((\wtilde q^{-1}))$, which is the completion of $k(\wtilde q)$ at the $\wtilde q^{-1}$-adic place.
The theorem above says that $T_q(\xi)$ cannot be the expansion
of a rational function in $k(\wtilde q)$.
In fact, if it was,
there would exists $c\in k(\wtilde  q)$ such that
$T_q(x)+c$ has a zero at $\xi$ and is solution of $\cL$.
This would imply that $\cL$ has a basis of solutions having a zero at $\xi$,
against the fact that the constants are solution of $\cL$.
\end{exa}

As in \cite{andreannalsII},
we can also deduce a Lindemann-Weierstrass type statement:

\begin{cor}
Let $K=k(\wtilde q)$, where $\wtilde q$ is a root of $q$.
We consider the $q$-exponential function $e_q(x)=\sum_{n\geq 0}\frac{x^n}{[n]_q^!}$
and  a set of element $a_1,\dots,a_r\in K$, which are
multiplicatively independent modulo $q^\Z$ (\ie $\a_1^\Z\cdots\a_r^\Z\cap q^\Z=\{1\}$).
Then the Laurent series $e_q(a_1\xi),\dots,e_q(a_r\xi)\in k((\wtilde q^{-1}))$
are linearly independent over $k(\wtilde q)$ for any $\xi \in K^\ast$.
\end{cor}

\begin{proof}
It is enough to notice that $e_q(a_1x),\dots,e_q(a_rx)$ is a basis of solutions of the
operator
$$
(\dq-a_1)\circ\dots\circ(\dq-a_r)\,.
$$
If there exist $\la_1,\dots,\la_r\in K$ such that
$\la_1 e_q(\a_1\xi)+\dots+\la_re_q(\a_r\xi)=0$, then $e_q(\a_i\xi)=0$ for any $i=1,\dots,r$,
because of Theorem \ref{thm:teogev}.
Since $e_q(x)$ satisfies the equation $y(qx)=(1+(q-1)x)e_q(x)$,
we deduce that $\xi\in\frac{q^{\Z_{\geq 1}}}{(1-q)\a_i}$, for any $i=1,\dots,r$.
The last assertion would imply that $\a_i\a_j^{-1}\in q^\Z$ for any pair of distinct $i,j$,
against the assumption.
\end{proof}

We can deduce by Theorem \ref{thm:teogev} an irrationality result
for all global $q$-Gevrey series $F(x)$ such that zero is not a slope of the
Newton Polygon at $\infty$ of the minimal $q$-difference operator that kills $F(x)$:

\begin{thm}\label{thm:irr}
Let $\ol{k(q)}$ be a fixed algebraic closure of $k(q)$
and $\wtilde K\subset\ol{k(q)}$ the maximal extension of $k(q)$ such that
the $q^{-1}$-adic norm of $k(q)$ extends uniquely to $\wtilde K$.
\par
Let $F(x)\in\wtilde K[[x]]\smallsetminus\wtilde K[x]$ be a global $q$-Gevrey series of
orders $(-s_1,-s_2)$,
with $(s_1,s_2)\in\cZ$, and
$\cL$ the minimal linear $q$-difference operator such that $\cL F(x)=0$.
We suppose that zero is not a slope of $\cL$ at $\infty$.
Then for all $\xi\in K^\ast$ the value $F(\xi)$ of the
$q^{-1}$-adic analytic entire function $F(x)$ is not an element
of $\wtilde K$ (but of its $\wtilde q^{-1}$-adic completion).
\end{thm}

Before proving the theorem, we give an example, which illustrates the proof:

\begin{exa}
Let us consider the $q$-analogue of a Bessel series
$$
B_q(x)=\sum_{n\geq 0}{x^n\over{[n]_q^!}^2}\,.
$$
The series $B_q(x)$ is solution of the
linear $q$-difference operator $(x\dq)^2-x$
that can be written also in the form:
$$
\cL=\sgq^2-2\sgq+(1-(q-1)^2x)\,.
$$
There is a unique factorization of a linear $q$-difference operator
linked to the slopes of its Newton Polygon (\cf \cite{sauloyfiltration}):
we deduce that $\cL$ is the minimal
$q$-difference operator killing $B_q(x)$ from the fact that the only slope of the
Newton-Polygon of $\cL$ at $\infty$ is $-1/2$.
We conclude that $B_q(\xi)=0$ for all $v$ such that $|q|_v>1$, with
$\xi\in{\mathbb P}^1(K)$,  implies $\xi=q^m/(q-1)^2$ for some integer $m\geq 2$.
\par
Let $K=k(\wtilde q)$, with $\wtilde q^r=q$ for some positive integer $r$.
In this case the $\wtilde q^{-1}$-adic norm is the only one such that $|q|_v>1$.
For any $c\in K$ we have:
$$
(q\sgq-1)\circ\cL(B_q(x)+c)=0\,.
$$
One notices that the slopes of the Newton Polygon of $(q\sgq-1)\circ\cL$ at $\infty$ are
$\{0,-1/2\}$, therefore we deduce from the uniqueness of the factorization that
$(q\sgq-1)\circ\cL$ is the minimal $q$-difference operator
killing $B_q(x)+c$.
Since constants are solutions of $(q\sgq-1)\circ\cL$,
Theorem \ref{thm:teogev} implies that no solution of $(q\sgq-1)\circ\cL$
can have a zero
at any point $\xi\in K^\ast$ as $\wtilde q^{-1}$-adic
holomorphic functions.
This means that the function $B_q(x)+c$ cannot have a zero as
a $\wtilde q^{-1}$-adic analytic function at
$\xi\in  K^\ast$, which means that
$B_q(x)$ takes values in $k((\wtilde q^{-1}))\smallsetminus k(\wtilde q)$ at each
$\xi\in K^\ast$.
\end{exa}

\begin{proof}[Proof of Theorem \ref{thm:irr}]
Let $c\in\wtilde K$, $c\neq 0$, $G(x)=F(x)+c$,
$\cL=\sum_{i=1}^\nu a_i(x)\dq^i\in\wtilde K[x,\dq]$ (resp.
$\cN=\sum_{j=1}^\mu b_j(x)\dq^j\in\wtilde K[x,\dq]$) be the minimal $q$-difference
operator  killing $F(x)$ (resp. $G(x)$).
Of course we may assume that $a_i(x),b_j(x)\in\wtilde K(x)$ and $a_\nu(x)=b_\mu(x)=1$,
and that everything is defined over a finite extension $K\subset\wtilde K$ of $k(q)$.
\par
Since:
$$
\l(\dq-{\dq(a_0)(x)\over a_0(x)}\r)\circ\cL(G(x))=0
\hbox{\ and }
\l(\dq-{\dq(b_0)(x)\over b_0(x)}\r)\circ\cN(F(x))=0\,,
$$
we must have $\nu-1\leq \mu\leq \nu+1$. Let us suppose first $\nu=\mu$.
Then
$$
\l(\dq-{\dq(a_0)(x)\over a_0(x)}\r)\circ\cL
=\l(\dq-{\dq(b_0)(x)\over b_0(x)}\r)\circ\cN
$$
since they have the same set of solutions and they are both monic operators.
By hypothesis, zero is not a slope of the Newton Polygon of $\cL$ at $\infty$, while
$\l(\dq-{\dq(a_0)(x)\over a_0(x)}\r)$ has only the zero slope at $\infty$: we
conclude by the uniqueness of the factorization that $\cL=\cN$.
We remark that the equality $\cL=\cN$ implies that constants are solutions of $\cL$ and that
$\cL$ has a zero slope at $\infty$, hence we obtain a contradiction.
So either $\mu=\nu-1$ or $\mu=\nu+1$.
If $\mu=\nu-1$, then
$$
\cL=\l(\dq-{\dq(b_0)(x)\over b_0(x)}\r)\circ\cN
$$
since both $\cL$ and $\cN$ are monic. Once again, constants are solution of $\cL$
and this is a contradiction. Finally, we have necessarily $\mu=\nu+1$ and
$$
\cN=\l(\dq-{\dq(b_0)(x)\over b_0(x)}\r)\circ\cL\,.
$$
Let us suppose that there exists $\xi\in K^\ast$,
such that $F(x)$
takes a value in $K$ at $\xi$, as
$\wtilde q^{-1}$-adic analytic function. Then all the solutions of $\cN$ would have a zero at
$\xi$ against the fact that the constants are solutions of $\cN$, hence $F(\xi)\neq 0$
is not in $K$.
\end{proof}


\providecommand{\bysame}{\leavevmode\hbox to3em{\hrulefill}\thinspace}
\providecommand{\MR}{\relax\ifhmode\unskip\space\fi MR }
\providecommand{\MRhref}[2]{%
  \href{http://www.ams.org/mathscinet-getitem?mr=#1}{#2}
}
\providecommand{\href}[2]{#2}


\begin{thebibliography}{DVRSZ03}

\bibitem[And89]{AGfunctions}
Yves Andr{\'e}, \emph{${G}$-functions and geometry}, Friedr. Vieweg \& Sohn,
  Braunschweig, 1989.

\bibitem[And00a]{andreannalsI}
\bysame, \emph{S\'eries {G}evrey de type arithm\'etique. {I}. {T}h\'eor\`emes
  de puret\'e et de dualit\'e}, Annals of Mathematics. Second Series
  \textbf{151} (2000), no.~2, 705--740.

\bibitem[And00b]{andreannalsII}
\bysame, \emph{S\'eries {G}evrey de type arithm\'etique. {II}. {T}ranscendance
  sans transcendance}, Annals of Mathematics. Second Series \textbf{151}
  (2000), no.~2, 741--756.

\bibitem[BB92]{BezivinBoutabaa}
Jean-Paul B{\'e}zivin and Abdelbaki Boutabaa, \emph{Sur les \'equations
  fonctionelles {$p$}-adiques aux {$q$}-diff\'erences}, Universitat de
  Barcelona. Collectanea Mathematica \textbf{43} (1992), no.~2, 125--140.

\bibitem[B{\'e}z92]{Beindex}
Jean-Paul B{\'e}zivin, \emph{Sur les \'equations fonctionnelles aux
  {$q$}-diff\'erences}, Aequationes Mathematicae \textbf{43} (1992), no.~2-3,
  159--176.

\bibitem[Bom81]{Bombieri}
Enrico Bombieri, \emph{On ${G}$-functions}, Recent progress in analytic number
  theory, Vol. 2 (Durham, 1979), Academic Press, London, 1981, pp.~1--67.

\bibitem[CC85]{CC1}
D.~V. Chudnovsky and G.~V. Chudnovsky, \emph{Applications of {P}ad\'e
  approximations to {D}iophantine inequalities in values of {$G$}-functions},
  Number theory ({N}ew {Y}ork, 1983--84), Lecture Notes in Math., vol. 1135,
  Springer, Berlin, 1985, pp.~9--51.

\bibitem[CC08]{connesconsani}
Alain Connes and Caterina Consani, \emph{On the notion of geometry over
  {$\mathbf F_1$}}, arXiv.org:0809.2926, 2008.

\bibitem[DGS94]{DGS}
Bernard Dwork, Giovanni Gerotto, and Francis~J. Sullivan, \emph{An introduction
  to {$G$}-functions}, Annals of Mathematics Studies, vol. 133, Princeton
  University Press, 1994.

\bibitem[DV00]{DVTesi}
Lucia Di~Vizio, \emph{Étude arithmétique des équations aux q-différences et des
  équations différentielles}, Ph.D. thesis, Université Paris 6, 2000.

\bibitem[DV02]{DVInv}
\bysame, \emph{Arithmetic theory of {$q$}-difference equations. {T}he
  {$q$}-analogue of {G}rothendieck-{K}atz's conjecture on {$p$}-curvatures},
  Inventiones Mathematicae \textbf{150} (2002), no.~3, 517--578,
  arXiv:math.NT/0104178.

\bibitem[DV04]{DVdwork}
\bysame, \emph{Introduction to {$p$}-adic {$q$}-difference equations (weak
  {F}robenius structure and transfer theorems)}, Geometric aspects of Dwork
  theory. Vol. I, II, Walter de Gruyter GmbH \& Co. KG, Berlin, 2004,
  arXiv:math.NT/0211217, pp.~615--675.

\bibitem[DVH09]{DivizioHardouin}
Lucia Di~Vizio and Charlotte Hardouin, \emph{Algebraic and differential generic
  galois groups}, preprint, 2009, arXiv:??

\bibitem[DVRSZ03]{gazette}
L.~Di~Vizio, J.-P. Ramis, J.~Sauloy, and C.~Zhang, \emph{\'{E}quations aux
  {$q$}-diff\'erences}, Gazette des Math\'ematiciens (2003), no.~96, 20--49.

\bibitem[DVZ07]{DVChanggui}
Lucia Di~Vizio and Changgui Zhang, \emph{On q-summation and confluence}, To
  appear in Annales de l'Insitut Fourier, 2007, arXiv:0709.1610.

\bibitem[GL05]{GarouHolonomy}
Stavros Garoufalidis and Thang T.~Q. L{\^e}, \emph{The colored {J}ones function
  is {$q$}-holonomic}, Geometry and Topology \textbf{9} (2005), 1253--1293.

\bibitem[GR90]{GR}
George Gasper and Mizan Rahman, \emph{Basic hypergeometric series},
  Encyclopedia of Mathematics and its Applications, vol.~35, Cambridge
  University Press, Cambridge, 1990, With a foreword by Richard Askey.

\bibitem[Har07]{HardouinIterative}
Charlotte Hardouin, \emph{{I}terative $q$-{D}ifference {G}alois {T}heory},
  preprint, 2007.

\bibitem[Kat70]{KatzTurrittin}
Nicholas~M. Katz, \emph{Nilpotent connections and the monodromy theorem:
  {A}pplications of a result of {T}urrittin}, Institut des Hautes {\'E}tudes
  Scientifiques. Publications Math{\'e}matiques (1970), no.~39, 175--232.

\bibitem[Man08]{ManinF1}
Yu.~I Manin, \emph{Cyclotomy and analytic geometry over {${\mathbf F}_1$}},
  arXiv:0809.1564, 2008.

\bibitem[MZ00]{MarotteZhang}
F.~Marotte and C.~Zhang, \emph{Multisommabilit\'e des s\'eries enti\`eres
  solutions formelles d'une \'equation aux {$q$}-diff\'erences lin\'eaire
  analytique}, Annales de l'Institut Fourier \textbf{50} (2000), no.~6,
  1859--1890.

\bibitem[Pra83]{Praag}
C.~Praagman, \emph{The formal classification of linear difference operators},
  Koninklijke Nederlandse Akademie van Wetenschappen. Indagationes Mathematicae
  \textbf{45} (1983), no.~2, 249--261.

\bibitem[Ram92]{RamisToulouse}
Jean-Pierre Ramis, \emph{About the growth of entire functions solutions of
  linear algebraic {$q$}-difference equations}, Toulouse. Facult\'e des
  Sciences. Annales. Math\'ematiques. S\'erie 6 \textbf{1} (1992), no.~1,
  53--94.

\bibitem[Sau00]{Sfourier}
Jacques Sauloy, \emph{Syst\`emes aux $q$-diff\'erences singuliers r\'eguliers:
  classification, matrice de connexion et monodromie}, Annales de l'Institut
  Fourier \textbf{50} (2000), no.~4, 1021--1071.

\bibitem[Sau04]{sauloyfiltration}
\bysame, \emph{La filtration canonique par les pentes d'un module aux
  {$q$}-diff\'erences et le gradu\'e associ\'e}, Annales de l'Institut Fourier
  \textbf{54} (2004), no.~1, 181--210.

\bibitem[Sou04]{Soulecar1}
Christophe Soul{\'e}, \emph{Les vari\'et\'es sur le corps \`a un \'el\'ement},
  Mosc. Math. J. \textbf{4} (2004), no.~1, 217--244, 312.

\bibitem[vdPS97]{vdPutSingerDifference}
Marius van~der Put and Michael~F. Singer, \emph{Galois theory of difference
  equations}, Springer-Verlag, Berlin, 1997.

\bibitem[Zha99]{ZhangFourier}
Changgui Zhang, \emph{D\'eveloppements asymptotiques {$q$}-{G}evrey et s\'eries
  {$Gq$}-sommables}, Annales de l'Institut Fourier \textbf{49} (1999), no.~1,
  227--261.

\end{thebibliography}
\end{document}